\documentclass[10pt,reqno]{amsart}
\usepackage{amsmath, amssymb}
\usepackage{amsfonts}
\usepackage{mathrsfs}
\usepackage{graphicx}
\usepackage[arrow,matrix,curve,cmtip,ps]{xy}

\makeatletter
 \def\@textbottom{\vskip \z@ \@plus 100pt}
 \let\@texttop\relax
\makeatother

\usepackage{amsthm}
  \usepackage{hyperref}
 \usepackage{upref}
\usepackage{a4wide}
\hypersetup{linkcolor=blue, colorlinks=true
,citecolor = red}
\allowdisplaybreaks

\newtheorem{thm}{Theorem}[section]
\newtheorem{prop}[thm]{Proposition}
\newtheorem{lem}[thm]{Lemma}
\newtheorem{cor}[thm]{Corollary}
\newtheorem{remark}[thm]{Remark}

\theoremstyle{definition}

\theoremstyle{remark}

\numberwithin{equation}{section}

\newcommand{\br}{\bar \rho}
\newcommand{\bt}{\bar \theta}
\newcommand{\ip}{(0,L)}
\newcommand{\sh}{\hat \sigma}
\newcommand{\vh}{\hat v}
\newcommand{\ph}{\hat \phi}

\begin{document}

\title[Linearized Compressible Navier- Stokes System]{Some Controllability Results For Linearized Compressible  Navier-Stokes system.}

 \date{\today}

\author{Debayan Maity}
\address{Centre for Applicable Mathematics, TIFR, 
Post Bag No. 6503, GKVK Post Office, Bangalore-560065,India.}
\email{debayan@math.tifrbng.res.in}
\thanks{The author is a member of an IFCAM-project, Indo-French Center for Applied Mathematics - UMI IFCAM,
Bangalore, India, supported by DST - IISc - CNRS - and Universit\'e Paul Sabatier Toulouse III.}




\begin{abstract}
In this article, we study the null controllability of linearized compressible Navier-Stokes system in one and two dimension.
We first study the one-dimensional compressible Navier-Stokes system for non-barotropic fluid 
linearized around a constant steady state. We prove that the linearized system around $(\br,0,\bt)$, with $\br > 0,$ $ \bt > 0$ is   
not null controllable by localized interior control or by boundary control. But the system is null controllable by interior controls
 acting everywhere in the velocity and temperature equation for regular initial condition. We also prove that the 
 the one-dimensional compressible Navier-Stokes system for non-barotropic fluid 
linearized around a constant steady state
$(\br,\bar v ,\bt)$, with $\br > 0,$  $\bar v > 0,$ $\bt > 0$ is not null controllable by localized interior control or by boundary control 
for small time $T.$ Next we consider two-dimensional compressible Navier-Stokes system for barotropic fluid linearized around a
constant steady state $(\br, {\bf 0}).$ We prove that this system is also not null controllable by localized interior control.
\end{abstract}
\maketitle









\section{Introduction}
Control of fluid flow has been an important area of research and has many practical applications.
The question about controllability of fluid
flows has attracted the attention of many researchers, more for incompressible flow but much less for compressible flow.
 In this paper we are interested in controllability properties of linearized compressible 
Navier-Stokes system.

For a compressible, isothermal
barotropic fluid (density is a function of only pressure), 
the Navier-Stokes system in $\Omega \subset \mathbb{R}^N,$ consists of equation of continuity
\begin{equation}
 \begin{array}{lll}
  \displaystyle
  \frac{\partial \rho}{\partial t}(x,t) + \mbox{div} [\rho(x,t){\bf u}(t,x)] = 0,
  \end{array}
  \end{equation}
  and the momentum equation
  \begin{equation}
  \begin{array}{lll} 
   \displaystyle
\rho(x,t)\left[\frac{\partial {\bf u}}{\partial t}(x,t) + ({\bf u}(x,t).\nabla) {\bf u}(x,t) - {\bf f}(x,t)\right]
 \\ [3.mm] \displaystyle
 = -\nabla p(x,t) + \mu \bigtriangleup {\bf u}(x,t) + (\lambda + \mu ) \nabla [\mbox{div } {\bf u}(x,t)],
  \end{array}
\end{equation}
where $\rho(x,t)$ is the density of the fluid, ${\bf u}(x,t)$ denotes the velocity vector in $\mathbb{R}^N$ and ${\bf f}(x,t)$ is an external force 
field in $\mathbb{R}^N.$ The pressure satisfies the following constitutive law
 \begin{equation}
 p(x,t) = a \rho^\gamma(x,t), \quad t > 0 , \quad x \in \Omega,
\end{equation}
for some constants $a > 0$, $\gamma \geq 1.$ The viscosity coefficients $\mu$ and $\lambda$ are assumed to be constant satisfying the following
thermodynamic restrictions, $\mu > 0$, $\lambda + \mu \geq 0.$ For non-barotropic fluid (when density is a function of pressure and temperature
of the fluid), 
the Navier-Stokes system consists of the equation of 
continuity, the momentum equation and an additional thermal energy equation 
\begin{equation}
 \begin{array}{lll}
 \displaystyle
 c_v \rho(x,t) \left[\frac{\partial \theta}{\partial t} + {\bf u}.\nabla \theta\right](x,t) + 
 \theta(x,t) \frac{\partial p}{\partial \theta}(x,t) \mbox{div } {\bf u}(x,t) 
 \\ [3.mm] \displaystyle
 = \kappa \bigtriangleup \theta (x,t) + \lambda (\mbox{div }{\bf u}(x,t))^2 + 
 2 \mu \sum_{i,j=1}^N \frac{1}{4} \left(\frac{\partial {\bf u}_i}{\partial x_j} + \frac{\partial {\bf u}_j}{\partial x_i}\right)^2,
 \end{array}
\end{equation}
where $\theta(x,t)$ denotes the temperature of the fluid, $c_v$ is the specific heat constant and 
$\kappa$ is the heat conductivity constant. For ideal gas,
the pressure  is given by Boyle’s law: 
\begin{equation}
p(x,t) = R \ \rho(x,t) \ \theta(x,t), 
\end{equation}
where $R$ is the universal gas constant (See \cite{F}).

In this article, we first consider
the compressible Navier-Stokes system for non-barotropic fluid  in a bounded interval $\ip,$
 linearized around a constant steady state $(\br,0,\bt),$ with $\br >0$ and $\bt > 0.$  More precisely we consider the system 
\begin{align}  \label{eq:linearized}
&\rho_t  + \br \;u_x \;\; = f \chi_{\mathcal{O}_1}, \mbox{ in }  (0,L) \times (0,T), \notag \\ 
&u_t - \frac{\lambda + 2\mu}{\br} u_{xx} + \frac{R \bt}{\br}\;\rho_x  + R \theta_x  \; = 
g \chi_{\mathcal{O}_2}, \mbox{ in }  (0,L) \times (0,T), \notag \\
&\theta_t -\frac{\kappa}{\br c_v} \theta_{xx} + \frac{R \bt}{c_v} u_x \; = h \chi_{\mathcal{O}_3}, 
\mbox{ in }  (0,L) \times (0,T),
\end{align}
where $\chi_{\mathcal{O}_i}$ is the characteristic function of an open subset ${\mathcal{O}}_i \subset \ip.$ We choose the following
initial and boundary conditions for the system \eqref{eq:linearized}:
\begin{equation} \label{eq:ini+bdy}
 \begin{array}{lll}
  \rho(0) = \rho_0,\quad u(0) = u_0 \mbox{ and } \theta(0) = \theta_0, \mbox{ in }  \ip,
  \\ [2.mm]
  u(0,t) = 0, \quad u(L,t) = 0 \quad \forall \ \ t > 0,
 \\ [2.mm]
 \theta(0,t) = 0, \quad \theta(L,t) = 0 \quad \forall \ \ t > 0.
 \end{array}
\end{equation}
In \eqref{eq:linearized} - \eqref{eq:ini+bdy}, $f$, $g$  and $h$ are distributed controls. 
We are interested in the following question: given $T > 0$ and $(\rho_0,u_0,\theta_0) \in (L^2(0,L))^3,$
can we find interior control functions such that the solution $(\rho,u,\theta)$ of \eqref{eq:linearized} - \eqref{eq:ini+bdy}
satisfies 
\begin{equation}
 (\rho,u,\theta)(x,T) = 0  \mbox{ for all } x \in \ip?
\end{equation}

Our first  main result regarding interior null controllability 
is the following,
\begin{thm} \label{thm1}
Let $${\mathcal{O}}_1  \subset \ip ,\quad {\mathcal{O}}_2 \subseteq \ip, \qquad {\mathcal{O}}_3 \subseteq \ip,$$ 
i.e., $\mathcal{O}_1$ is a proper subset of $\ip.$
Let us assume 
$(\rho_0,u_0,\theta_0) \in L^2(0,L) \times L^2(0,L) \times L^2(0,L)$. 
The system \eqref{eq:linearized}
- \eqref{eq:ini+bdy} is not null controllable in any $T > 0 $ by the interior controls $f \in L^2(0,T;L^2({\mathcal{O}}_1))$, 
$g \in L^2(0,T;L^{2}({\mathcal{O}}_2))$ and $h \in L^2(0,T;L^{2}({\mathcal{O}}_3))$ acting on density, velocity and temperature
equation respectively. 
\end{thm}

\begin{remark}
 The above negative result can be extended to the case of less regular interior controls or boundary control. See Remark \ref{rem:lessRegControl} and 
 Remark \ref{rem:bdyControl} for more details.
\end{remark}

Our next positive  result shows that, if initial density $\rho_0$ lies in a more regular space then the linearized system is null controllable by 
velocity and temperature controls acting everywhere in the domain. 
\begin{thm} \label{thm2}
Let $f \equiv 0$ in \eqref{eq:linearized}. Let us denote $$H^1_m(0,L) = \left\{ \rho \in H^1(0,L) | \int_0^1 \rho(x) dx = 0 \right\}.$$ 
Let us assume
$(\rho_0,u_0,\theta_0) \in H^1_m(0,L) \times L^2(0,L) \times L^2(0,L)$. Then for any $T > 0,$ there exist controls  $g \in L^2(0,T;L^2(0,L))$ and 
$h \in L^2(0,T;L^2(0,L))$  acting everywhere in the velocity and temperature equation respectively, such that the solution of \eqref{eq:linearized} - \eqref{eq:ini+bdy} satisfies
\[
 (\rho,u,\theta)(x,T) = 0 \mbox{ for all } x\in \ip.
\]
\end{thm}

The next result shows that the above result is sharp as null controllability cannot be achieved by
localized interior 
 velocity and temperature controls.

\begin{thm} \label{thm3}
 Let $f \equiv 0$ in \eqref{eq:linearized}. 
 Let  $${\mathcal{O}}_2 \subset \ip , \qquad {\mathcal{O}}_3 \subseteq \ip,$$ i.e., $\mathcal{O}_2$ is a proper subset of $\ip.$
 Let us assume
$(\rho_0,u_0,\theta_0) \in H^1_m(0,L) \times L^2(0,L) \times L^2(0,L)$. 
The system \eqref{eq:linearized}
- \eqref{eq:ini+bdy} is not null controllable in any $T > 0 $ by the interior controls  
$g \in L^2(0,T;L^2({\mathcal{O}}_2))$ and $h \in L^2(0,T;L^2({\mathcal{O}}_3)),$ acting on velocity and temperature equation respectively.
\end{thm}


The proof of these results relies on the observability inequality. We know that the null controllability of a linear system is 
equivalent to a certain
observability inequality for the solutions of adjoint system (see \cite{CORON}, Chapter 2). To prove the negative results, we 
will construct particular solutions for the adjoint system such that the observability inequality cannot hold. In order to do that 
first we will consider the adjoint system in $\mathbb{R} \times (0,T)$ as a terminal value problem. We will construct highly localized
solutions known as ``Gaussian Beam''. Similar kind of construction has been used for hyperbolic equations by Ralston (\cite{ralston}) and
for wave equations by
Maci\`{a} and Zuazua (\cite{Macia-Zuazua}). We will 
prove that  solutions are localized in a small neighbourhood of any $x_0 \in \mathbb{R}.$ Thus given an observation set, we can
always find an interval away from the observation set such that the solutions are localized in that interval. Using this we are able to
prove the negative results. To the author's best knowledge these are new  results regarding controllability issues of Navier-Stokes
system for non-barotropic fluid. 

In Theorem \ref{thm3}, we proved a positive result when  controls acting everywhere in the equation. The question then arises: whether positive
results could be obtained by using control supported in a small, but moving region, as in \cite{MRR,SRZ}. Rosier and Rouchon in \cite{RR} proved 
that the structurally damped wave equation in one dimension is  not null controllable by a boundary control. Later on 
Martin, Rosier and Rouchon in \cite{MRR} proved that  the same equation in one dimension with periodic boundary conditions,
 is null controllable with a moving distributed control for sufficiently large time.
Chaves-Silva, Rosier and Zuazua in \cite{SRZ} extend the above result to higher dimension. 
The structure
of the system considered by the authors in \cite{MRR,SRZ}, in some sense, is similar to the linearized compressible, barotropic
Navier-Stokes system in one dimension as well as in higher dimension. 
These issues will be discussed in a future work which is in  progress. 

We have studied in this paper the null controllability of the linearized compressible Navier-Stokes system only. The ``Gaussian Beam''
construction is used to show negative results. However one may use other techniques based on the use of nonlinearity 
(see \cite{CORON} for example) to achieve controllability results for the full nonlinear system.

There have been some results regarding the control of compressible barotropic fluid models in recent years.
Amosova in \cite{AMO} considers compressible Navier-Stokes system for 
viscous barotropic fluid in one dimension in Lagrangian coordinates in a bounded domain 
 $(0, 1)$ with Dirichlet boundary condition. She proves local
exact controllability to trajectories for the velocity in any time $T > 0,$ using a localized interior control
 on the velocity equation, provided that the initial density is
already on the targeted trajectory  and initial condition lies in $H^1(0,1) \times H^1_0(0,1).$

Ervedoza, Glass, Guerrero and Puel in \cite{EGGP} consider the compressible Navier-Stokes system in
one space dimension in a bounded domain $(0, L)$. They prove local exact controllability to constant
states $(\br , \bar u )$ with $\bar \rho > 0, \bar u \neq  0$ using two boundary controls, both for density and velocity, in time
$\displaystyle T > \frac{L}{|\bar u|}$ when initial condition lies in $H^3(0,L) \times H^3(0,L).$

Chowdhury, Ramaswamy and Raymond in \cite{CRR} consider the  compressible barotropic 
Navier-Stokes system linearized around a constant steady state $(Q_0 , 0)$ with $Q_0 > 0$ in a bounded domain $(0,\pi)$. 
They proved that the linearized system
is not null controllable by a localized control or by boundary control. They also proved that the linearized system is null controllable
by an interior control acting everywhere in 
the velocity equation when initial condition lies in $H^1_m(0,\pi) \times L^2(0,\pi).$ 

Chowdhury in \cite{SC} considers the  compressible barotropic 
Navier-Stokes system linearized about a constant steady state $(Q_0 , V_0)$ with $Q_0 > 0,  V_0 > 0$ in $(0,L)$ with 
Dirichlet boundary condition and an interior control on the velocity equation acting on open subset $(0,l) \subset (0,L).$  
He proves that the system is approximately controllable in $L^2(0,L) \times L^2(0,L)$ when $\displaystyle T > \frac{L - l}{V_0}.$ 
He also proves a similar result in 
two dimension.

Chowdhury, Mitra, Ramaswamy and Renardy in \cite{CMRR2} consider the  compressible barotropic 
Navier-Stokes system linearized about a constant steady state $(Q_0 , V_0)$ with $Q_0 > 0,  V_0 > 0$ in $(0,2\pi)$ with 
periodic boundary condition. They proved that the linearized system is null controllable by a localized velocity control when 
$\displaystyle T > \frac{2\pi}{V_0}$ and initial condition lies in $H^1_m(0,2\pi) \times L^2(0,2\pi).$ 

Our linearized system \eqref{eq:linearized} - \eqref{eq:ini+bdy} is similar to the linearized system considered by the 
authors in \cite{CRR}. So we expect similar controllability results. Their method is based on explicit expression for 
eigenfunctions and the behaviour of the spectrum of the
linearized operator. They proved that there is an accumulation point in the spectrum of the linearized operator. 
This system behaves very badly with respect to controllability properties and a similar type of 
controllability behaviour is also observed in \cite{RR,MS,MR} for different types
of systems where an accumulation point is present in the spectrum of linearized operator.
But the method used in \cite{CRR} does not seem to fit very well in our case. In fact one can prove that there is an accumulation point
in the spectrum of the linearized operator considered here, for certain 
boundary condition. But the expressions of eigenvalues and eigenfunctions are complicated. 
So here we use Gaussian Beam approach to achieve the negative results. This technique does not require the knowledge of the spectrum and it 
seems to extend to higher dimension also. 

The controllability properties  are completely different, if we consider compressible Navier-Stokes system
linearized around non null velocity. For barotropic fluid, the system linearized around $(Q_0,0)$ is not controllable in any time $T$
by localized interior control but the system linearized around $(Q_0, V_0)$ is null controllable by localized interior control for large
time $T.$ It is interesting to note that there is no accumulation point in the spectrum of the linearized operator in the latter case and 
better controllability behaviour at least for $T$ large enough(see \cite{CMRR2}). 
But the question remains what
happens when time $T$ is small enough. Our results answer this question in the negative.

We consider the compressible non-barotropic Navier-Stokes equation linearized around constant steady state 
$(\br, \bar v, \bt),$ $\br > 0, \bar v > 0, \bt > 0$ 
\begin{equation} \label{eq:blinearized1}
\begin{array}{l}
\displaystyle
 \rho_t  + \bar v \rho_x  + \br \;u_x \;\; = f \chi_{\mathcal{O}_1},  \mbox{ in } (0,L) \times (0,T), \\ [2.mm]
\displaystyle
 u_t - \frac{\lambda + 2\mu}{\br} u_{xx}  + \frac{R \bt}{\br}\;\rho_x  + \bar v u_x +
R \theta_x  \; = g \chi_{\mathcal{O}_2},  \mbox{ in } (0,L) \times (0,T), \\[3.mm]
\displaystyle
\theta_t (x,t) -\frac{\kappa}{\br c_v} \theta_{xx} + \frac{R \bt}{c_v} u_x + \bar v \theta_x \; = h \chi_{\mathcal{O}_3},
 \mbox{ in } (0,L) \times (0,T),
\\[3.mm] \displaystyle
\rho(0) = \rho_0,\quad \quad   u(0)= u_0,  \mbox{ and } \quad \theta(0) = \theta_0, \quad \mbox{ in }  (0,L),
\\ [2.mm] \displaystyle
\rho(0,t) = 0, \;\;\; u(0,t) = 0 = u(L,t), \;\;  \forall \ t \in (0,T),
\\ [2.mm] \displaystyle
\theta(0,t) = 0 = \theta(L,t), \;\;  \forall \ t \in (0,T).
\end{array}
\end{equation}
We prove the following theorem. 
\begin{thm} \label{thm1.4}
Let $${\mathcal{O}}_1 = (l_1,l_2) \subset \ip, \quad  {\mathcal{O}}_2 \subseteq \ip \quad {\mathcal{O}}_3 \subseteq \ip,$$ 
i.e., $\mathcal{O}_1$ is a proper subset of $\ip.$ Let us assume 
$(\rho_0,u_0,\theta_0) \in L^2(0,L) \times L^2(0,L) \times L^2(0,L)$. 
If  
$\displaystyle T < \max\left\{\frac{l_1}{\bar v}, \frac{L-l_2}{\bar v} \right\},$ then
the system \eqref{eq:blinearized1} is not null controllable by localized interior controls
 $f \in L^2(0,T;L^2(\mathcal{O}_1)),$ $g \in L^2(0,T;L^{2}({\mathcal{O}}_2))$ and $h \in L^2(0,T;L^{2}({\mathcal{O}}_3))$ 
 acting on density, velocity and temperature equation respectively. 
\end{thm}

 As a corollary of the above Theorem, one can rule out null controllability of compressible barotropic Navier-Stokes system linearized around
 constant steady state $(\br, \bar v)$ in small time $T,$ using a boundary control or localized interior control. 
\begin{cor}
 We consider compressible barotropic Navier-Stokes system 
linearized around a constant steady state $(\br, \bar v)$, $\br > 0, \bar v > 0$ in $(0,L) \times (0,T)$ or 
in $(0,2\pi) \times (0,T)$ as in \cite{CMRR2,SC}.
\begin{itemize}
\item[(i)] For initial condition belonging to  $L^2(0,L) \times L^2(0,L),$ the system with Dirichlet boundary condition is not null 
controllable at any time $T > 0$ by a interior control acting only in the velocity equation. The control may act in a non empty open subset
of $(0,L)$ or in the whole domain $(0,L).$
\item [(ii)] The  system with periodic boundary condition
 is not null controllable by interior control localized in $(l_1,l_2) \subset (0,2\pi),$ acting only in the velocity equation when initial
 condition lies in $H^1_m(0,2\pi) \times L^2(0,2\pi)$ and time
 $\displaystyle T < \max\left\{\frac{l_1}{\bar v}, \frac{2\pi-l_2}{\bar v} \right\}.$ 
 \item [(iii)] For initial condition belonging to $L^2(0,L) \times L^2(0,L),$ the same system is not null controllable by boundary control 
 if time $\displaystyle T < \frac{L}{\bar v}.$
\end{itemize}
\end{cor}

\begin{remark}
From the above corollary we see that the condition $T > \displaystyle \frac{L}{|\bar v|}$ in Ervedoza,Glass,Guerrero and Puel (\cite{EGGP})
is natural.
\end{remark}

Next we will show that, our method can be extended to higher dimension also. For simplicity we consider  
 the compressible barotropic Navier-Stokes system in two-dimensional bounded domain $\Omega$, linearized around a constant steady 
state solution $(\bar \rho, 0, 0)$ , $\br > 0,$ 
\begin{equation} \label{eq:linearized2D}
 \begin{array}{l}
\displaystyle
\rho_t +  \bar \rho \ \mathrm{div}\;{\bf u} \;\;= f \chi_{\mathcal{O}_1  }, \mbox{ in } \Omega \times (0,T), 
\\[2.mm] \displaystyle
{\bf u} - \frac{\mu}{\br}  \Delta {\bf u} - \frac{\lambda + \mu}{\br}\nabla[\mathrm{div} \ {\bf u} ]
+\;a \gamma \br^{\gamma-2} \nabla \rho \;=\;{\bf g}  \chi_{{\mathcal O}_2} ,  \mbox{ in } \Omega \times (0,T),
\\[2.mm] \displaystyle
\rho(0) = \rho_0 \quad \mbox{and}\quad   {\bf u}(0)= {\bf u}_0, \quad  \mbox{ in } \Omega,
\\ [2.mm] \displaystyle
 {\bf u} = {\bf 0}  \ \ \mbox{ on } \partial \Omega \times (0,T).
\end{array}
\end{equation}
where ${\bf u} = (u_1,u_2)$  and $\mathcal{O}_1, \mathcal{O}_2$ are open subsets of $\Omega.$ 
We obtain the following negative
null controllability result for the system \eqref{eq:linearized2D}.
\begin{thm} \label{thm1.5}
 Let $$\mathcal{O}_1 \subset \Omega, \quad \mathcal{O}_2 \subseteq \Omega,$$
 i.e., $\mathcal{O}_1$ is a proper open subset of $\Omega.$ Let us assume that $(\rho_0, {\bf u}_0) \in (L^2(\Omega))^3.$ Then 
 the system \eqref{eq:linearized2D} is not null controllable  in time any $T > 0$  by interior controls $f \in L^2(0,T;L^2(\mathcal{O}_1))$ 
 and ${\bf g}  \in (L^2(0,T;L^2(\mathcal{O}_2)))^2.$
\end{thm}

The plan of the paper is as follows. In section 2, we study the  control system linearized around a constant steady state $(\br,0,\bt)$ in 
one dimension. We prove 
Theorem \ref{thm1}, Theorem \ref{thm2} and Theorem \ref{thm3} here. In section 3, we study the  control system linearized around a constant 
steady state
$(\bar \rho, \bar v , \bt).$ Theorem \ref{thm1.4} is proved here. In section 4 we consider the linearized system in two 
dimension around constant 
steady state $(\br, {\bf 0}).$ We prove Theorem \ref{thm1.5} here.

{\bf Acknowledgement}: The author would like to thank Prof. Sylvain Ervedoza for providing important references on Gaussian Beams.
The author also would like to thank him and Prof. Mythily Ramaswamy for
very useful discussions which improved the initial version. The author acknowledges the financial support under the project "PDE Control"  from the Indo French
Centre for Applied Mathematics (IFCAM).

\section{Null Controllability of Compressible Non-Barotropic Navier Stokes System in
One Dimension Linearized about
\texorpdfstring{$(\br,0,\bt)$}{Lg} }

In this section, we will discuss interior null controllability of the system \eqref{eq:linearized} - \eqref{eq:ini+bdy}.
We introduce the positive constants 
 \begin{equation} \label{eq:constant}
 \nu_0 := \frac{\lambda+2\mu}{\br} ,\quad \quad k_0 := \frac{\kappa}{\br c_v}, \quad \quad b := \frac{\bt}{c_v}.
\end{equation}
Let us define  $Z = L^2(0,L) \times L^2(0,L) \times L^2(0,L)$ endowed with the inner product 
\begin{equation*}
 \left \langle
\begin{pmatrix}\rho \\ u \\ \theta \end{pmatrix},\begin{pmatrix}\sigma \\ v \\ \phi \end{pmatrix} \right \rangle_{Z} :=
R \bt \int_0^{L} \rho(x) \sigma(x)\;dx + \br^2 \int_0^{L} \;u(x) v (x) dx + 
\frac{\br^2 c_v}{\bt} \int_0^{L} \;\theta(x)  \phi(x) dx.
\end{equation*}
The following proposition about existence and uniqueness of the system \eqref{eq:linearized} - \eqref{eq:ini+bdy}  follows
easily from semigroup theory.
\begin{prop}
 Let $(\rho_0,u_0,\theta_0) \in Z$. Let us assume that $f \in L^2(0,T;L^2(\mathcal{O}_1))$, \\ $g \in L^2(0,T;L^{2}(\mathcal{O}_2))$
 and $h \in L^2(0,T;L^{2}(\mathcal{O}_3))$. Then \eqref{eq:linearized} - \eqref{eq:ini+bdy} has a unique solution $(\rho,u,\theta)$
 with $\rho \in L^2(0,T;L^2(0,L))$, $u \in L^2(0,T;H^1_0(0,L))$ and $\theta \in L^2(0,T;H^1_0(0,L))$. Moreover $(\rho,u,\theta) \in C([0,T];Z).$
\end{prop}

\subsection{Observability Inequality}

The idea is to use the
adjoint system to derive certain identity  which can be used to obtain an observability
inequality, equivalent to null controllability. (See \cite{CORON}, Chapter 2). For this we consider the following
adjoint problem, 
 \begin{equation}  \label{eq:adjoint1}
 \begin{array}{lll}
-  \sigma_t  - \br \;v_x \;\; = 0, \mbox{ in } \ip \times (0,T),\\ [2.mm]
\displaystyle
 - v_t - \nu_0 v_{xx} - \frac{R \bt}{\br}\;\sigma_x - R \phi_x \; = 0, \mbox{ in } \ip \times (0,T), \\[3.mm]
\displaystyle
 - \phi_t  - k_0 \phi_{xx} - \frac{R \bt}{c_v} v_x \; = 0,\ \mbox{ in } \ip \times (0,T),
\\ [2.mm] \displaystyle
  \sigma(T) = \sigma_T,\quad v(T) = v_T ,\quad \phi(T) = \phi_T, \mbox{ in } \ip ,
  \\ [2.mm]
  v(0,t) = 0 = v(L,t) \quad \forall \ \ t > 0,\quad
 \phi(0,t) = 0 = \phi(L,t) \quad \forall \ \ t > 0.
\end{array}
\end{equation}
with $(\sigma_T,v_T,\phi_T) \in Z.$  The adjoint system \eqref{eq:adjoint1} is also well posed in $Z$. In fact we have 
\begin{prop}
 Let $(\sigma_T,v_T,\phi_T) \in Z.$ The system \eqref{eq:adjoint1} has a unique solution with
 $\sigma \in L^2(0,T;L^2(0,L))$, $v \in L^2(0,T;H^1_0(0,L))$ and $\phi \in L^2(0,T;H^1_0(0,L))$. Moreover $(\sigma,v,\phi)$ 
 belongs to $ C([0,T];Z).$
\end{prop}

Let us first assume that $(\rho_0, u_0, \theta_0),$  $(\sigma_T,v_T, \phi_T) \in {\mathcal C}_c^\infty(0,L) \times 
{\mathcal C}_c^\infty(0,L) 
\times {\mathcal C}_c^\infty(0,L),$ \\ 
$f \in {\mathcal C}_c^\infty((0,T) \times {\mathcal{O}_1}),$ $g \in {\mathcal C}_c^\infty((0,T) \times {\mathcal{O}_2}),$
$h \in {\mathcal C}_c^\infty((0,T) \times {\mathcal{O}_3})$ and 
 let $(\rho,u,\theta)$ and $(\sigma, v, \phi)$ be
the solutions of \eqref{eq:linearized} and \eqref{eq:adjoint1} respectively. Taking inner product in $Z$ of 
\eqref{eq:linearized} with $(\sigma, v, \phi)$ and integrating we obtain 
\begin{align*}
 R \bt \int_0^T \int_0^L (\partial_t \rho + \bar \rho u_x) \sigma \; dx dt+ 
 \bar \rho^2  \int_0^T \int_0^L (\partial_t u  - \nu_0 u_{xx} + \frac{R\bt}{\br} \rho_x + R \theta_x) v \; dx dt  + \\
 \frac{\br^2 c_v}{\bt} \int_0^T \int_0^L (\partial_t \theta - k_0 \theta_{xx} + \frac{R \bt}{c_v} u_x)\phi \; dx dt
 = R \bt \int_0^T \int_{\mathcal{O}_1} f \sigma \; dx dt \  
 \\
 + \br^2 \int_0^T \int_{\mathcal{O}_2}  g v  \ dx dt  + \frac{\br^2 c_v}{\bt} \int_0^T \int_{\mathcal{O}_2}  h \phi \ dx dt. 
\end{align*}
An integration by parts and use of \eqref{eq:adjoint1} gives 
\begin{align} \label{eq:identity_0}
 R \bt  \int_0^L [\rho(x, T) \sigma_T(x) - \rho_0(x) \sigma(x,0) ] dx + 
  \br^2  \int_0^L [u(x, T) v_T(x) - u_0(x) v(x,0)]  dx    \notag 
  \\ + \frac{\br^2 c_v}{\bt} \int_0^L [\theta(x, T) \phi_T(x) - \theta_0(x) \phi(x,0)]  dx 
  = R \bt \int_0^T \int_{\mathcal{O}_1} f \sigma \; dx dt \  \notag
 \\
 + \br^2 \int_0^T \int_{\mathcal{O}_2} g v \ dx dt  + \frac{\br^2 c_v}{\bt} \int_0^T \int_{\mathcal{O}_2}  h \phi  \ dx dt. 
\end{align}

The above  relation leads us to the identity equivalent to null controllability.
\begin{prop}
 For each initial state $(\rho_0, u_0,\theta_0) \in Z$, 
 the solution of the system \eqref{eq:linearized} - \eqref{eq:ini+bdy} can be driven to rest by  interior controls
 $f \in L^2(0,T;L^2({\mathcal O_1}))$, $g \in L^2(0,T;L^{2}({\mathcal O_2}))$ and $h \in L^2(0,T;L^{2}({\mathcal O_3}))$ in time $T$
 if and only if 
 \begin{align}    \label{eq:identity_1}
 R \bt \int_0^T \int_{\mathcal{O}_1} f \sigma \; dx dt \  
 + \br^2 \int_0^T \int_{\mathcal{O}_2} g  v  \ dx dt  + \frac{\br^2 c_v}{\bt} \int_0^T \int_{\mathcal{O}_3} h \phi  \ dx dt \notag
 \\ 
 + \left \langle  \begin{pmatrix}\rho_0 \\ u_0 \\ \theta_0  \end{pmatrix},\begin{pmatrix}
 \sigma(\cdot,0) \\ v(\cdot,0) \\ \phi(\cdot,0) \end{pmatrix} \right \rangle_Z
 = 0 
 \end{align}
 for all $(\sigma_T,v_T,\phi_T) \in Z,$
 where $(\sigma,v,\phi)$ is the solution of the adjoint system \eqref{eq:adjoint1}. 
 \end{prop}
\begin{proof}
 By a density argument we deduce that for any $(\rho_0,u_0,\theta_0) \in Z$ and $(\sigma_T, v_T, \phi_T) \in Z $ the identity 
 \eqref{eq:identity_0} holds. Thus from \eqref{eq:identity_0}, it follows that  \eqref{eq:identity_1} holds if and only if 
 \eqref{eq:linearized} - \eqref{eq:ini+bdy} is null controllable and $f,g,h$ are the corresponding controls. 
 \end{proof}

One can use the identity \eqref{eq:identity_1} to get an observability inequality which is also equivalent to  null controllability.
More precisely we have the following Proposition (See \cite{SZ} Section 2 and \cite{CORON} Chapter 2 ).  
\begin{prop}
 The system \eqref{eq:linearized} - \eqref{eq:ini+bdy} is null controllable in $Z$ in time $T > 0$ if and only if there exists a constant $C$
 such that for any terminal condition $(\sigma_T, v_T,\phi_T) \in Z$, $(\sigma, v, \phi)$, the solution of the adjoint 
 problem \eqref{eq:adjoint1}
 satisfies the following observability inequality
\begin{align} \label{eq:observability_ineq_1}
   &\|\sigma(\cdot,0)\|^2_{L^2(0,L)} + \|v(\cdot,0)\|^2_{L^2(0,L)}  + \|\phi(\cdot,0)\|^2_{L^2(0,L)}  
  \notag \\
  &\leq C \left ( \int_0^T \int_{{\mathcal O}_1} \sigma^2 dx dt + \int_0^T \int_{{\mathcal O}_2} v^2 dx dt
  + \int_0^T \int_{{\mathcal O}_3} \phi^2 dx dt \right ).
\end{align}
\end{prop}

\subsection{Highly Localized Solutions}

We now want to prove that the system \eqref{eq:linearized} - \eqref{eq:ini+bdy} is not null controllable in $Z$ when 
controls are localized. Our idea is to show that the observability inequality 
\eqref{eq:observability_ineq_1} does not hold in this case. 
For this we first consider the adjoint problem in whole real line :
\begin{equation} \label{eq:adjoint2}
\begin{array}{lll}
-\sigma_t  - \br \;v_x \;\; = 0, \mbox{ in } \mathbb{R} \times (0,T), \\ [2.mm]
\displaystyle
 - v_t - \nu_0 v_{xx}  - \frac{R \bt}{\br}\;\sigma_x  - R \phi_x  \; = 0, \mbox{ in } \mathbb{R} \times (0,T), 
   \\[3.mm]
\displaystyle
 - \phi_t (x,t) - k_0 \phi_{xx} - \frac{R \bt}{c_v} v_x \; = 0, \mbox{ in } \mathbb{R} \times (0,T),
\\ [2.mm] \displaystyle
  \sigma(T) = \sigma_T,\quad v(T) = v_T ,\quad \phi(T) = \phi_T \quad \mbox{ in } \mathbb{R}.
\end{array}
\end{equation}

First we will construct a particular solution of the above adjoint problem which is localized in a small neighbourhood of 
any $x_0 \in \mathbb{R}.$
For this we would like to have a Fourier representation formula for the solution of \eqref{eq:adjoint2}. Let us assume that 
$(\sigma_T,v_T,\phi_T) \in (L^2(\mathbb{R}))^3$ and $(\sigma, v, \phi) \in (L^2(0,T;L^2(\mathbb{R})))^3.$
We define $\hat \sigma (\xi,t)$ the Fourier transform of $\sigma$ in the space variable $x$ for each $t \in (0,T)$ as follows,
\[
 \hat \sigma (\xi,t) = \int_{\mathbb R} \sigma(x,t) e^{-i \xi x} dx , \quad \xi \in {\mathbb R}. 
\]
We also define $\hat v (\xi,t)$ and $\hat \phi(\xi,t)$ in a similar manner. Applying the Fourier transform in \eqref{eq:adjoint2}, we obtain
the following system of ODE satisfied by 
$\sh(\xi,t)$, $\vh(\xi,t)$ and  $\ph(\xi,t)$ 

\begin{equation} \label{eq:adjoint-Fr}
 \begin{array}{lll}
  - \partial_t \sh (\xi,t) - \br \ (i\xi) \ \vh(\xi,t)\;\; = 0,   \qquad \xi \in {\mathbb R} , t \in (0,T),\\ [2.mm]
\displaystyle
 - \partial_t \vh(\xi,t)- \nu_0  (- \xi^2) \vh(\xi,t) - \frac{R \bt}{\br}\;(i\xi)\sh(\xi,t) - R (i\xi)\ph(\xi,t) \; = 0, 
 \qquad \xi \in {\mathbb R} , t \in (0,T),  \\[3.mm]
\displaystyle
 - \partial_t \ph (\xi,t) - k_0 (- \xi^2)\ph (\xi,t) - \frac{R \bt}{c_v}(i\xi) \vh(\xi,t) \; = 0, \qquad \xi \in {\mathbb R} , t \in (0,T),
\\ [2.mm] \displaystyle
  \sh(\xi,T) = \sh_T(\xi),\quad \vh(\xi,T) = \vh_T(\xi) ,\quad \ph(\xi,T) = \ph_T(\xi) \quad \xi \in {\mathbb R}.
 \end{array}
\end{equation}
Let us define 
\begin{equation} \label{eq:matrix}
\displaystyle
 A(\xi) = \begin{pmatrix} 
                        \displaystyle  0                       & \br i \xi               & 0 \\
                        \displaystyle  \frac{R\bt}{\br} i \xi  & - \nu_0 \xi^2                 & Ri\xi \\
                        \displaystyle   0                      &\displaystyle \frac{R \bt}{c_v} i \xi & -k_0\xi^2
                        \end{pmatrix}
\end{equation}
Then \eqref{eq:adjoint-Fr} can be written in the following form 
\begin{align}
&- \begin{pmatrix} \sh \\ \vh \\ \ph \end{pmatrix}_t = A(\xi) \begin{pmatrix} \sh \\ \vh \\ \ph \end{pmatrix} \notag \\
&(\sh,\vh,\ph)(\xi,T) = (\sh_T(\xi), \vh_T(\xi), \ph_T(\xi)) .
\end{align}
The unique solution of the above system of ODE can be written as 
\begin{equation} \label{eq:ODEsoln}
 (\sh, \vh, \ph ) (\xi,t) = e^{A(\xi)(T-t)}(\sh_T, \vh_T, \ph_T).
\end{equation}

We will now discuss some properties of the eigenvalues of $A(\xi).$ Let $a:\mathbb{R} \rightarrow \mathbb{R}$, $b:\mathbb{R} \rightarrow \mathbb{R}$ and $c:\mathbb{R} \rightarrow \mathbb{R}$ be three 
smooth functions. Let $\lambda^3 + a(\xi) \lambda^2 + b(\xi) \lambda + c(\xi) $ be a cubic polynomial. 
Let us define the  discriminant of the above cubic polynomial
 $$ D(\xi) = 18a(\xi)b(\xi)c(\xi) - 4a^3(\xi)c(\xi) + a^2(\xi)b^2(\xi) - 4b^3(\xi) - 27 c^2(\xi).$$
 Now the roots of the above
 cubic polynomial are given by the following formula
 \begin{align} \label{eq:formula}
 \lambda_k(\xi) =\displaystyle -\frac{1}{3}\left(a(\xi) + \omega_k C(\xi) + \frac{D_0(\xi)}{\omega_k C(\xi)}\right), \ \mbox{ for } k=1,2,3. 
 \end{align}
 where
 $$\omega_1 = 1, \quad \displaystyle \omega_2 = \frac{-1+i\sqrt{3}}{2}, \quad \omega_3 = \frac{-1 - i\sqrt{3}}{2}$$
 are the three cubic roots of unity, and 
 $$\displaystyle C(\xi) = \sqrt[3]{\frac{D_1(\xi) + \sqrt{D_1(\xi)^2 - 4 D_0(\xi)^3}}{2}}  \qquad 
$$
with
$$ D_0(\xi) = a^2(\xi) - 3 b(\xi), \quad D_1(\xi) = 2 a^3(\xi)-9 a(\xi) b(\xi) + 27 c(\xi) \ \mbox{and } D_1(\xi)^2 - 4 D_0(\xi)^3 = -27 D(\xi)$$
and when 
$$D_0(\xi) \neq 0 \ \ \mbox{and } D(\xi) \neq 0.$$
In this formula, $\sqrt{~~}$ and $\sqrt[3]{~~}$ denote any choice for the square or cube roots, but one has to 
be consistent with the choice for all $\xi$. 

If $D(\xi) \neq 0$ and $D_0(\xi) = 0$ for some $\xi,$ the sign of $\sqrt{D_1(\xi)^2 - 4D_0(\xi)^3} = \sqrt{D_1(\xi)^2}$ 
has to be chosen to have $C(\xi) \neq 0,$ i.e. one should define $\sqrt{D_1(\xi)^2} = D_1(\xi).$ In this case the roots are given by 
$$\lambda_k(\xi) = -\frac{1}{3}\left(a(\xi) + \omega_k \sqrt[3]{D_1(\xi)}\right), \mbox{ for } k=1,2,3.$$

If $D(\xi) = 0$ and $D_0(\xi) = 0$ for some $\xi,$ the three roots are  equal 
$$ \lambda_1(\xi) = \lambda_2(\xi) = \lambda_3(\xi) = -b(\xi)/3.$$

If $D(\xi) = 0$ and $D_0(\xi) \neq 0$ for some $\xi,$ there is a double root 
$$ \lambda_1(\xi) = \lambda_2(\xi)  = \frac{9c(\xi) - a(\xi) b(\xi)}{2 D_0(\xi)},$$
and a simple root 
$$\lambda_3 (\xi) = \frac{4 a(\xi) b(\xi) - 9 c(\xi) - a(\xi)^3}{D_0(\xi)}.$$

As $a,$ $b$ and $c$ are differentiable functions
of $\xi,$ it is easy to deduce that real part and the complex part of the roots are also differentiable for all $\xi \in 
\{\xi \in \mathbb{R} | C(\xi) \neq 0, D(\xi) \neq 0, D_0(\xi) \neq 0\}.$
We have the following lemma about the properties of the eigenvalues of $A(\xi).$

\begin{lem} \label{lem:2.5}
 The eigenvalues of $A(\xi)$ always have non positive real part for all $\xi \in \mathbb{R}.$
 Let $\{ -\lambda(\xi), \, -\mu(\xi), \, -\delta(\xi) \} $, be
  the eigenvalues of $A(\xi)$, where $\mathrm{Re} \ \lambda(\xi) \geq 0,\mathrm{Re} \ \mu(\xi) \geq 0$ and $\mathrm{Re} \ \delta(\xi) \geq 0$.
 There exists a constant
 $\xi_0 = \xi_0(R,\bt,\nu_0,k_0,b) >0,$ such that
 for $|\xi| \geq \xi_0,$ one of the eigenvalues, say $\delta(\xi)$  satisfies
\begin{equation}   \label{eq:lim}
\lim_{|\xi| \to \infty}   \delta(\xi)  = \omega_0 ,
\end{equation}
where $\omega_0 = \frac{R\bt}{\nu_0} .$ 
\end{lem}

\begin{proof}
 The eigenvalues of $A(\xi)$ are given by the roots of the characteristic polynomial
 \begin{equation} \label{eq:characteristic}
\lambda^3 + (\nu_0 + k_0)\xi^2 \lambda^2 + (R\bt \xi^2 + R^2 b \xi^2 + k_0 \nu_0 \xi^4)\lambda + k_0R\bt \xi^4 = 0.  
 \end{equation}
Since it is a polynomial of degree three, it will always have a real root.

Now the polynomial $\lambda^3 + a\lambda^2 + b \lambda + c$ is stable i.e. all the roots have negative real part if and only if 
$$a > 0, \ b > 0, \ c > 0  \ \mbox{ and } \ a b > c.$$ (Theorem 2.4, Part I of \cite{ZA}).
For our case, 
$$a(\xi) = (\nu_0 + k_0)\xi^2, \quad b(\xi) = R\bt \xi^2 + R^2 b \xi^2 + k_0 \nu_0 \xi^4, \quad c(\xi)  = k_0 R \bt \xi^4.$$ 
And it is easy to verify that
$a(\xi) \ b(\xi) > c(\xi)$ for all $\xi \in \mathbb{R} \setminus \{0\}$. So all the eigenvalues have negative real part for all 
$\xi \in \mathbb{R} \setminus \{0\}$. For $\xi = 0,$ the characteristic polynomial has only one root $0.$

Let $\{ -\lambda(\xi), \, -\mu(\xi), \, -\delta(\xi) \} $, be 
the eigenvalues of $A(\xi),$ given by the formula \eqref{eq:formula}. Thus 
$$\mathrm{Re} \ \lambda(\xi) \geq 0,\quad \mathrm{Re} \ \mu(\xi) \geq 0,
\quad \mathrm{Re} \ \delta(\xi) > 0 \mbox{ for all } \xi \in \mathbb{R}\setminus \{0\}.$$

Comparing the coefficients of the characteristic polynomial we obtain,
\begin{align} \label{eq_3.12}
\begin{cases}
&\lambda(\xi) + \mu(\xi) + \delta(\xi)  = (\nu_0 + k_0)\xi^2,\\
&\lambda(\xi)\mu(\xi) + \mu(\xi)\delta(\xi) + \delta(\xi)\lambda(\xi) = R\bt \xi^2 + R^2 b \xi^2 + k_0 \nu_0 \xi^4, \\
&\lambda(\xi) \mu(\xi) \delta(\xi) = k_0R\bt \xi^4 .
\end{cases}
\end{align}

It is well known that, if the discriminant $D > 0,$ then the polynomial has three distinct real roots and if $D < 0,$ then the equation has 
one real root and two complex conjugate roots. In our case the discriminant is 
\begin{align}
D(\xi) = k_0^2 \ \nu_0^2 (\nu_0 - k_0)^2 \xi^{12} + \left [ 18(\nu_0 + k_0)k_0^2\nu_0 R \bt  - 4 (\nu_0 + k_0)^3 k_0 R\bt \right. \notag \\
\left. + 2 (\nu_0 + k_0)^2 (R\bt + R^2 b) k_0 \nu_0   - 12 (R\bt + R^2 b) k_0^2 \nu_0^2 \right] \xi^{10} + O(\xi^8).
\end{align}
We also have
$$ D_0(\xi) = (\nu_0 + k_0)\xi^4 - 3 (R\bt \xi^2 + R^2 b \xi^2 + k_0 \nu_0 \xi^4).$$
This leads us to consider the following two cases. 

{\bf Case I.} Let $\nu_0 \neq k_0.$ In this case, there exists a positive constant $\xi_0,$ such that $D(\xi) > 0$  and 
$D_0(\xi) \neq 0 $ for all $|\xi| \geq \xi_0.$ 
Hence for  $|\xi| \geq \xi_0,$  $\lambda(\xi),\mu(\xi)$ and $\delta(\xi) $ are all real, positive and distinct. 
Let us define, 
\begin{align}
 \tilde \lambda(\xi) = \frac{\lambda(\xi)}{\xi^2}, \tilde \mu (\xi) = \frac{\mu(\xi)}{\xi^2}, \tilde \delta (\xi) = \frac{\delta(\xi)}{\xi^2}.
\end{align}

In \eqref{eq_3.12}, letting $|\xi| \rightarrow \infty,$ we obtain 
\begin{align} \label{eq:2.15}
\begin{cases}
&\tilde\lambda(\xi) + \tilde\mu(\xi) + \tilde\delta(\xi)  = (\nu_0 + k_0),\\
&\displaystyle\lim_{|\xi| \to \infty} \left(\tilde\lambda(\xi)\tilde\mu(\xi) + \tilde\mu(\xi)\tilde\delta(\xi) + 
\tilde\delta(\xi)\tilde\lambda(\xi)\right) =  k_0 \nu_0, \\
&\displaystyle\lim_{|\xi| \to \infty} \tilde\lambda(\xi) \tilde\mu(\xi) \tilde\delta(\xi) = 0.
\end{cases}
\end{align}
As  $\lambda(\xi),\mu(\xi),\delta(\xi) $ are all positive, from the first equation of \eqref{eq:2.15}, we deduce that 
$\tilde\lambda(\xi),\tilde\mu(\xi),\tilde\delta(\xi) $ are all bounded. They are also continuous for $|\xi| \geq \xi_0.$
As $\lambda(\xi),\mu(\xi),\delta(\xi)$ are all distinct and continuous for $|\xi| \geq \xi_0,$ without loss of generality 
we assume that 
$$ \tilde \lambda(\xi) > \tilde \mu(\xi) > \tilde \delta(\xi)\ \mbox{ for } |\xi| \geq \xi_0.$$
From the 
last equation of \eqref{eq:2.15}, we obtain  
that  $\tilde\delta(\xi)$ converges to $0$ as $|\xi| \rightarrow \infty.$ 
From \eqref{eq:2.15}
we obtain, 
\begin{align} 
\lim_{|\xi| \to \infty}\left(\tilde\lambda(\xi) + \tilde\mu(\xi)\right)  = (\nu_0 + k_0),
\qquad \lim_{|\xi| \to \infty} \left(\tilde\lambda(\xi) \ \tilde\mu(\xi)\right) =  k_0 \nu_0.
\end{align}
Therefore, $\tilde\lambda(\xi)$ and $\tilde\mu(\xi)$ both converge as $|\xi|\rightarrow \infty$ and one of them converges 
to $\nu_0$ and the other one 
to $k_0.$ Without loss of generality we assume that 
 \[
 \lim_{|\xi| \to \infty}\frac{\lambda(\xi)}{\xi^2} = \nu_0; \quad \lim_{|\xi| \to \infty}\frac{\mu(\xi)}{\xi^2} = k_0.
\]

Therefore
\[
 \lim_{|\xi| \to \infty} \lambda(\xi) = \infty , \qquad \lim_{|\xi| \to \infty} \mu(\xi) = \infty.
\]
From the last two equations of \eqref{eq_3.12}, we have 
\begin{align} \label{eq:2.17}
 \frac{1}{\lambda(\xi)} + \frac{1}{\mu(\xi)} + \frac{1}{\delta(\xi)} = \frac{1}{k_0 \xi^2}+ \frac{Rb}{k_0\bt \xi^2} + \frac{\nu_0}{R\bt}.
\end{align}
Letting $|\xi| \rightarrow \infty,$ we deduce that $\displaystyle \lim_{|\xi| \to \infty} \delta(\xi) = \frac{R\bt}{\nu_0} := \omega_0.$

{\bf Case II.} Let $\nu_0 = k_0.$ In this case, 
\begin{align*}
 D(\xi) = -4 k_0^4 R^2 b \xi^{10} + O(\xi^8).
\end{align*}
Thus there exists a constant $\xi_0,$ such that $D(\xi) < 0$ and $D_0(\xi) \neq 0$ for all $|\xi|\geq \xi_0.$ Hence for all $|\xi|\geq \xi_0,$ we have one real root and
two complex conjugate roots. Let 
\begin{align*}
 \lambda(\xi) = \alpha(\xi) + i \beta(\xi) \mbox{ and }  \mu(\xi) = \alpha(\xi) - i \beta(\xi). 
\end{align*}
Hence, for all $|\xi| \geq \xi_0,$  $\alpha(\xi)$ and $\delta(\xi)$ are all real and positive. 
Let us define, 
\begin{align}
 \tilde \alpha(\xi) = \frac{\alpha(\xi)}{\xi^2}, \tilde \beta (\xi) = \frac{\beta(\xi)}{\xi^2}, \tilde \delta (\xi) = \frac{\delta(\xi)}{\xi^2}.
\end{align}

From \eqref{eq_3.12}  we obtain 
\begin{align} \label{eq:2.20}
\begin{cases}
&2\tilde\alpha(\xi) +  \tilde\delta(\xi)  = 2 k_0,\\
& \displaystyle \tilde\alpha^2(\xi) + \tilde\beta^2(\xi) + 2 \tilde\alpha(\xi)\tilde\delta(\xi) =  k_0^2 + \frac{R\bt + R^2 b}{\xi^2}, \\
&\displaystyle \left(\tilde\alpha^2(\xi) + \tilde\beta^2(\xi)\right) \tilde\delta(\xi) = \frac{k_0R\bt}{\xi^2}.
\end{cases}
\end{align}
As  $\alpha(\xi),\beta(\xi),\delta(\xi) $ are all positive, from the first two equations of \eqref{eq:2.20}, we deduce that 
$\tilde\alpha(\xi),\tilde\beta(\xi),$ $\tilde\delta(\xi) $ are all bounded for $|\xi|\geq \xi_0$. They are also continuous. From the 
last equation of \eqref{eq:2.20}, we obtain 
$$ \displaystyle\lim_{|\xi| \to \infty} \tilde\alpha^2(\xi) \tilde\delta(\xi) = 0 , \quad  
\displaystyle\lim_{|\xi| \to \infty} \tilde\beta^2(\xi) \tilde\delta(\xi) = 0.$$
Multiplying the second equation of \eqref{eq:2.20} by $\tilde \alpha(\xi) \tilde \delta(\xi)$ and letting $|\xi|\rightarrow \infty$ we obtain
$$\displaystyle 
\lim_{|\xi|\rightarrow \infty} \tilde \alpha(\xi) \tilde \delta(\xi)  = \lim_{|\xi|\rightarrow \infty} \frac{1}{k_0^2}
\left[\tilde\alpha(\xi) \tilde \delta(\xi)(\tilde\alpha^2(\xi) + \tilde\beta^2(\xi)) + 2 \tilde\alpha(\xi)^2\tilde\delta(\xi)^2 
 - \tilde\alpha(\xi) \tilde \delta(\xi)\frac{R\bt + R^2 b}{\xi^2}\right] = 0.$$

Again multiplying the second equation of \eqref{eq:2.20} by $\tilde \delta(\xi)$ and letting $|\xi|\rightarrow \infty$ we obtain
$$\displaystyle 
\lim_{|\xi|\rightarrow \infty}  \tilde \delta(\xi)  = \lim_{|\xi|\rightarrow \infty} \frac{1}{k_0^2}
\left[ \tilde \delta(\xi)(\tilde\alpha^2(\xi) + \tilde\beta^2(\xi)) + 2 \tilde\alpha(\xi)\tilde\delta(\xi)^2 
 -  \tilde \delta(\xi)\frac{R\bt + R^2 b}{\xi^2}\right] = 0.$$
Now we can proceed as in Case I, to obtain
$\lim_{|\xi| \rightarrow \infty} \delta(\xi) = \omega_0.$


\end{proof}

\begin{figure}[ht!]
\centering
\includegraphics[width=70mm,height = 40mm]{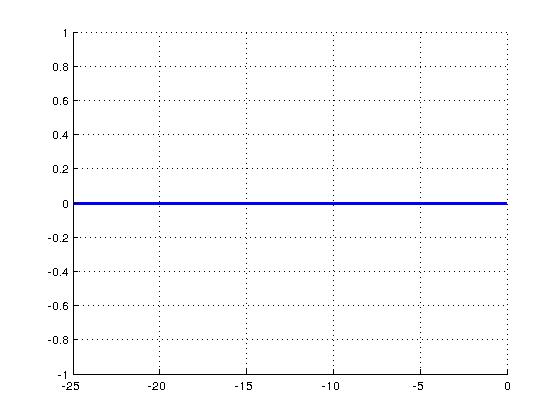}
%
\includegraphics[width=70mm,height = 40mm]{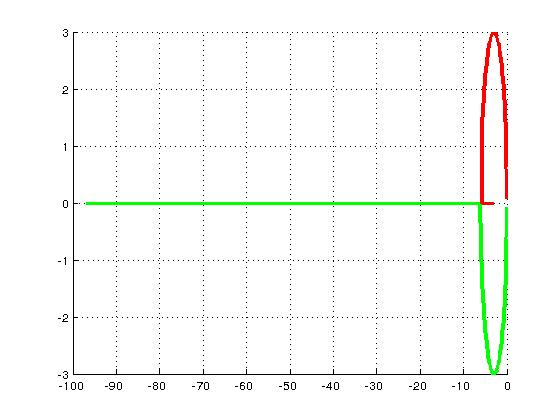}
\caption{Eigenvalues of $A(\xi)$ in the complex plane for $\xi$ varies in the interval $(0,5)$ 
when $\nu_0 = \bar \theta = b = 1, R =3 $ and $k_0 = 4.$ Blue line represents $-\lambda(\xi)$,
green line represents $-\mu(\xi)$ and red line represents $-\delta(\xi).$
}
\label{overflow2}
\end{figure}

In fact more can be said about $\delta(\xi).$ We have 
\begin{lem} \label{lem:2.6}
 Let $A(\xi)$ be  defined as in \eqref{eq:matrix} and  $-\delta(\xi)$ be the eigenvalue satisfying 
 $\displaystyle \lim_{|\xi|\rightarrow \infty} \delta(\xi) = \omega_0.$
 Then there exists a constant $0 < a_1$ such that for all $\xi \in \mathbb{R},$ $$0 \leq \mathrm{Re} \ \delta(\xi) < a_1.$$
 Further for $|\xi| \geq \xi_0,$ where $\xi_0$ as in Lemma \ref{lem:2.5}, $\delta(\xi)$ is differentiable and we have 
 \begin{equation} \label{eq:2.13}
  \left|\frac{d}{d \xi} \delta(\xi)\right| \leq \frac{C}{|\xi|},
 \end{equation}
for some positive constant $C$. 
 \end{lem}
 \begin{proof}
  From Lemma \ref{lem:2.5}, $\delta(\xi)$ lies on the right half side of the complex plane and it is bounded for all $\xi \in \mathbb{R}.$ Thus 
  there exists a constant $0< a_1,$ such that $$0 \leq \mathrm{Re} \ \delta(\xi) < a_1.$$
  The coefficients of the characteristic polynomial \eqref{eq:characteristic} are differentiable and from the formula \eqref{eq:formula}
  it is easy to deduce that, the real and complex  parts of the roots are differentiable for $|\xi| \geq \xi_0.$
  For $|\xi| > \xi_0,$ $\delta(\xi)$ is real and hence differentiable.
 As 
  $-\delta(\xi)$ is a root of the characteristic equation \eqref{eq:characteristic}, we have
  \[
- \delta(\xi)^3 + (\nu_0 + k_0)\xi^2 \delta(\xi)^2 - (R\bt \xi^2 + R^2 b \xi^2 + k_0 \nu_0 \xi^4)\delta(\xi) + k_0R\bt \xi^4 = 0    
  \]
Differentiating the equation with respect to $\xi$ and using the fact that $\delta(\xi)$ is bounded for all $\xi$ we obtain the estimate
\eqref{eq:2.13}.
 \end{proof}

\begin{lem} \label{lem:eigenfunction}
Let $A(\xi)$ be defined as in \eqref{eq:matrix} and $-\delta(\xi)$ is the eigenvalue satisfying, 
 $\displaystyle \lim_{|\xi|\rightarrow \infty} \delta(\xi) = \omega_0.$ 
 The eigenfunction of $A(\xi)$ corresponding to $-\delta(\xi)$ is $\displaystyle\left(1,\frac{i \delta(\xi)}{\br \xi}, d_\delta(\xi)\right)$
 where \\
 $$ \displaystyle d_\delta(\xi) = -\frac{\delta(\xi)^2 - \nu_0 \xi^2 \delta(\xi) + R \bt \xi^2}{R \br \xi^2}.$$
 For $|\xi|$ sufficiently large, we have 
 $$|d_{\delta}(\xi)| \leq \frac{C}{|\xi|^2} .$$
 for some positive constant $C$. 
\end{lem}
\begin{proof}
From \eqref{eq:2.17}, we first obtain 
\begin{align*}
 \xi^2(-\nu_0 \delta(\xi) + R\bar \theta) = R\bar\theta \delta(\xi)\left(\frac{1}{k_0} + \frac{Rb}{k_0\bar\theta} - 
 \frac{1}{\lambda(\xi) / \xi^2} - \frac{1}{\mu(\xi) / \xi^2} \right).
\end{align*}
Hence for $|\xi|$ sufficiently large, we obtain 
\begin{align*}
 |\xi^2(-\nu_0 \delta(\xi) + R\bar \theta)| \leq C,
\end{align*}
and the estimate of $d_\delta(\xi)$ follows .
\end{proof}

\begin{figure}[ht!]
\centering
\includegraphics[width=70mm,height = 40mm]{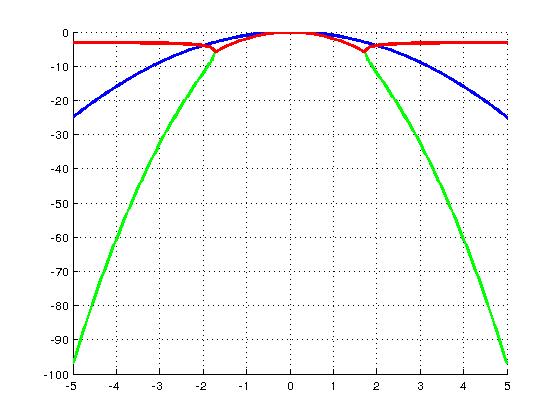}
\caption{In the horizontal axis we represent $\xi$  and the vertical one the real part of  
eigenvalues of $A(\xi)$  for $\xi$ varies in the interval  $(-5,5)$ 
when $\nu_0 = \bar \theta = b = 1, R =3 $ and $k_0 = 4.$ Blue line represents $\mathrm{Re}(-\lambda(\xi))$,
green line represents $\mathrm{Re}(-\mu(\xi))$ and red line represents $\mathrm{Re}(-\delta(\xi)).$ From the figure 
it is clear that the branch corresponding to 
$\delta(\xi)$ is of hyperbolic type, while the other two are parabolic type.
}
\label{overflow3}
\end{figure}

We now want to give a representation formula for  solution of \eqref{eq:adjoint2}.
We have the following proposition. 
\begin{prop}
 Let us consider
 \begin{equation} \label{eq:terminal+fr}
\displaystyle (\sh_T, \vh_T, \ph_T) = \sh_T(\xi)\left(1,\frac{i \delta(\xi)}{\br \xi}, d_\delta(\xi)\right),
 \end{equation}
 for a suitable $\sh_T \in L^2(\mathbb{R}),$
such that $(\sh_T, \vh_T, \ph_T) \in (L^2(\mathbb{R}))^3.$ Then 
\begin{align} \label{eq:representation}
 & \sigma(x,t) = \frac{1}{2 \pi} \int_{\mathbb{R}} \sh_T(\xi) e^{i x \xi} e^{-\delta(\xi) (T-t)} \ d\xi, \notag \\
 & v(x,t) = \frac{1}{2\pi} \int_{\mathbb{R}} \sh_T(\xi) e^{i x \xi} \frac{i \delta(\xi)}{\br \xi} e^{-\delta(\xi) (T-t)} \ d\xi, \notag \\
& \phi(x,t) = \frac{1}{2\pi} \int_{\mathbb{R}} \sh_T(\xi) e^{i x \xi} d_{\delta}(\xi) e^{-\delta(\xi) (T-t)} \ d\xi,
 \end{align}
 is the solution of \eqref{eq:adjoint2}, corresponding to the terminal condition
\begin{align}  \label{eq:adjoint_condn}
 & \sigma_T(x) = \frac{1}{2 \pi} \int_{\mathbb{R}} \sh_T(\xi) e^{i x \xi}  \ d\xi, \notag \\
 & v_T(x) = \frac{1}{2\pi} \int_{\mathbb{R}} \sh_T(\xi) e^{i x \xi} \frac{i \delta(\xi)}{\br \xi}  \ d\xi, \notag \\
& \phi_T(x) = \frac{1}{2\pi} \int_{\mathbb{R}} \sh_T(\xi) e^{i x \xi} d_{\delta}(\xi)  \ d\xi.
 \end{align} 
\end{prop}
\begin{proof}
Denote $(\sh,\vh,\ph)$ to be the solution of \eqref{eq:adjoint-Fr} corresponding to the terminal condition defined in \eqref{eq:terminal+fr}. 
Then from \eqref{eq:ODEsoln}, the solution of \eqref{eq:adjoint-Fr} can be written as 
\begin{equation} \label{eq:representation2}
 \sh(\xi, t) = e^{-\delta(\xi) (T-t)} \sh_T,\quad \vh(\xi, t) = e^{-\delta(\xi) (T-t)}  \frac{i \delta(\xi)}{\br \xi} \sh_T, \quad 
 \ph(\xi, t) = e^{-\delta(\xi) (T-t)} d_\delta(\xi) \sh_T.
\end{equation}
It is easy to verify that  $(\sh,\vh,\ph) \in L^2(0,T;L^2(\mathbb{R}))^3.$
By taking the inverse of Fourier transform, we obtain, $(\sigma, v, \phi)$ defined as in  \eqref{eq:representation}, is the solution of system \eqref{eq:adjoint2} corresponding to
the terminal condition \eqref{eq:adjoint_condn}.
\end{proof}

We are now interested in the construction of some particular solutions of the equation \eqref{eq:adjoint2}, which are localized 
in a neighbourhood of some $x_0 \in \mathbb{R}$. For that, we choose $\sh_T^\epsilon$ in \eqref{eq:terminal+fr} using a suitable cut off function 
$\psi$ and analyze the solution $(\sigma^\epsilon,  v^\epsilon, \phi^\epsilon)$ of \eqref{eq:adjoint-Fr} given by \eqref{eq:representation}.
\begin{thm} \label{thm4}
 Let $\psi$ be a smooth function compactly supported in $(0,1)$ and of unit $L^2$ norm. For any $\epsilon > 0$, sufficiently small 
 and for any $x_0 > 0$,
 let us take 
 \begin{equation}
 \displaystyle
  \sh_T^\epsilon(\xi) = \epsilon^{\frac{1}{4}} \psi\left(\sqrt{\epsilon}\left(\xi - \frac{1}{\epsilon}\right)\right) e^{-i x_0 \xi}.
 \end{equation}
Let  $(\sigma_T^\epsilon(x),v_T^\epsilon(x),\phi_T^\epsilon(x))$  and 
$(\sigma^\epsilon(x,t), v^\epsilon(x,t),\phi^\epsilon(x,t))$ be as in \eqref{eq:adjoint_condn} 
and  \eqref{eq:representation} respectively. Then they satisfy the following
\begin{itemize}
 \item[(i)] $\frac{1}{2\pi} e^{-2 a_1 T} \leq \|\sigma^\epsilon(\cdot,0)\|^2_{L^2(\mathbb{R})} 
 \leq \frac{1}{2\pi} $,
 \item[(ii)] For any $\eta > 0$ ,  there exists a constant $C$ independent of $\epsilon$ such that 
 \begin{equation}
  \|\sigma^\epsilon(\cdot,t)\|^2_{L^2(0,T;L^2(|x - x_0|\geq \eta))} \leq C \sqrt{\epsilon} ,
 \end{equation}
\item[(iii)]
For some positive constant $C$ independent of $\epsilon$, 
\begin{equation}
 \|v^\epsilon\|^2_{L^2(0,T;L^2(\mathbb{R}))} \leq C \epsilon^{2} , \qquad 
 \|\phi^\epsilon\|^2_{L^2(0,T;L^2(\mathbb{R}))} \leq C \epsilon^{4} .
\end{equation}
\end{itemize}
\end{thm}

\begin{proof}
Let us denote 
\begin{equation} \label{eq:terminal+fr+eps}
 \displaystyle (\sh^\epsilon_T, \vh^\epsilon_T, \ph^\epsilon_T)(\xi) = 
\sh^\epsilon_T(\xi) \left(1,\frac{i \delta(\xi)}{\br \xi}, d_\delta(\xi)\right)
\end{equation}
First we verify that $(\sh^\epsilon_T, \vh^\epsilon_T, \ph^\epsilon_T) \in (L^2(\mathbb{R}))^3.$ We have 
\begin{align*}
 \int_{\mathbb{R}} |\sh^\epsilon_T(\xi)|^2 \ d \xi = \int_{\mathbb{R}} \epsilon^{1/2} 
 \left|\psi\left(\sqrt{\epsilon}\left(\xi - \frac{1}{\epsilon}\right)\right)\right|^2 d\xi =
 \int_{\mathbb{R}} |\psi(\zeta)|^2 \ d\zeta = 1.
\end{align*}
Thus $\sh_T^\epsilon \in L^2(\mathbb{R})$. Now 
\begin{align*}
 \int_{\mathbb{R}} |\vh^\epsilon_T(\xi)|^2 \ d \xi 
 &= \int_{\mathbb{R}} \epsilon^{1/2} 
 \left|\psi\left(\sqrt{\epsilon}\left(\xi - \frac{1}{\epsilon}\right)\right)\right|^2 \frac{\delta(\xi)^2}{\br^2 \xi^2} d\xi \\
 &= \frac{\epsilon^2}{\br^2} \int_{\mathbb{R}} |\psi(\zeta)|^2 
 \frac{\delta(\frac{\zeta}{\sqrt{\epsilon}}+ \frac{1}{\epsilon})^2}{(\zeta\sqrt{\epsilon}+1)^2}\ d\zeta \leq C \epsilon^2.
\end{align*}
So we have $\vh_T^\epsilon \in L^2(\mathbb{R})$. Similarly we can show that $\ph_T^\epsilon \in L^2(\mathbb{R}).$ Thus the representation
formula \eqref{eq:representation} is well defined. Now we will prove (i).

Note that
 $$ \sigma^\epsilon(x,0) = \frac{1}{2 \pi} \int_{\mathbb{R}} \sh^\epsilon_T(\xi) e^{i x \xi} e^{-\delta(\xi) T} \ d\xi .$$
 By Parseval's relation we have 
 $$\|\sigma^\epsilon(\cdot,0)\|^2_{L^2(\mathbb{R})} = \frac{1}{2\pi} \int_{\mathbb{R}} e^{-2 \delta(\xi) T} |\sh^\epsilon_T(\xi)|^2 \ d \xi .$$
Using Lemma \ref{lem:2.6} we have 
\begin{align*}
 \frac{1}{2\pi} e^{-2 a_1 T}\int_{\mathbb{R}} |\sh^\epsilon_T(\xi)|^2 \ d \xi \leq 
 \frac{1}{2\pi} \int_{\mathbb{R}} e^{-2 \delta(\xi) T} |\sh^\epsilon_T(\xi)|^2 \ d \xi \leq
  \frac{1}{2\pi} \int_{\mathbb{R}} |\sh^\epsilon_T(\xi)|^2 \ d \xi.
\end{align*}
Hence we have proved (i). To prove (ii) we have using a change of variable formula
\begin{align*}
 \sigma^\epsilon(x,t) &= \frac{1}{2 \pi}\int_{\mathbb{R}} \epsilon^{\frac{1}{4}} \psi(\sqrt{\epsilon}(\xi - \frac{1}{\epsilon}))
 e^{i (x-x_0) \xi} e^{-\delta(\xi)(T-t)} \ d\xi \\ 
 &= \frac{\epsilon^{-1/4}}{2\pi} \int_{\mathbb{R}} \psi(\zeta) e^{i(x-x_0)(\frac{\zeta}{\sqrt{\epsilon}}+ \frac{1}{\epsilon})}
 e^{-\delta(\frac{\zeta}{\sqrt{\epsilon}}+ \frac{1}{\epsilon})(T-t)} \ d\zeta.
\end{align*}
Note that 
\[
\displaystyle \frac{d}{d \zeta}  e^{i(x-x_0)(\frac{\zeta}{\sqrt{\epsilon}}+ \frac{1}{\epsilon})}  = 
\frac{i(x-x_0)}{\sqrt{\epsilon}} e^{i(x-x_0)(\frac{\zeta}{\sqrt{\epsilon}}+ \frac{1}{\epsilon})}.
\]
Thus for $|x-x_0| \geq \eta > 0,$ we have, after integration by parts
\begin{align}
 \sigma^\epsilon(x,t) &=  \frac{\epsilon^{1/4}}{2\pi i (x-x_0)} \int_0^1 
 \frac{d}{d \zeta} \left( e^{i(x-x_0)(\frac{\zeta}{\sqrt{\epsilon}}+ \frac{1}{\epsilon})} \right) \psi(\zeta)
  e^{-\delta(\frac{\zeta}{\sqrt{\epsilon}}+ \frac{1}{\epsilon})(T-t)} \ d\zeta , \notag \\
  &= \frac{-\epsilon^{1/4}}{2\pi i (x-x_0)} \int_0^1 
   e^{i(x-x_0)(\frac{\zeta}{\sqrt{\epsilon}}+ \frac{1}{\epsilon})}  
   \frac{d}{d \zeta} 
   \left( \psi(\zeta) e^{-\delta(\frac{\zeta}{\sqrt{\epsilon}}+ \frac{1}{\epsilon})(T-t)} \right) \ d\zeta.
\end{align}
Now 
\begin{align}
 \frac{d}{d \zeta} 
   \left( \psi(\zeta) e^{-\delta(\frac{\zeta}{\sqrt{\epsilon}}+ \frac{1}{\epsilon})(T-t)} \right) = 
    e^{-\delta(\frac{\zeta}{\sqrt{\epsilon}}+ \frac{1}{\epsilon})(T-t)} \left(\psi'(\zeta) -\psi(\zeta) 
    \delta'\left(\frac{\zeta}{\sqrt{\epsilon}}+ \frac{1}{\epsilon} \right)\frac{1}{\sqrt{\epsilon}}(T-t) \right).
\end{align}
Thus for $\epsilon$ small, we have for some $C$ independent of $\epsilon$
\begin{align*}
 \left|\frac{d}{d \zeta} 
   \left( \psi(\zeta) e^{-\delta(\frac{\zeta}{\sqrt{\epsilon}}+ \frac{1}{\epsilon})(T-t)} \right) \right| 
   \leq C\left(1 + \frac{\epsilon}{\zeta\sqrt{\epsilon}+1} \frac{1}{\sqrt{\epsilon}}\right) \leq C (1 + \sqrt{\epsilon}).
\end{align*}
Therefore, for $\epsilon$ small enough and $|x - x_0|\geq \eta >0,$ we have 
\begin{align}
 |\sigma^\epsilon(x,t)| \leq C \frac{\epsilon^{1/4}}{2\pi|x - x_0|} (1+ \sqrt{\epsilon}) \leq C \frac{\epsilon^{1/4}}{|x-x_0|}.
\end{align}
Thus there exists a positive constant $C(T,\eta)$, such that 
\begin{align}
 \|\sigma^\epsilon\|^2_{L^2(0,T;L^2(|x-x_0|\geq \eta))} \leq C \sqrt{\epsilon}.
\end{align}
This proves (ii). 
We have
\begin{align*}
 \int_0^T \int_{\mathbb{R}} |v^\epsilon(x,t)|^2 \ dx dt  &= \frac{1}{2 \pi} \int_0^T \int_{\mathbb{R}} |\vh^\epsilon(\xi,t)|^2 \ d\xi dt =
  \frac{1}{2\pi} \int_0^T \int_{\mathbb{R}} \left|e^{-\delta(\xi)(T-t)} \frac{\delta(\xi)}{\br\xi}  \sh^\epsilon_T(\xi) \right|^2 \ d\xi dt, \\
  &\leq \frac{\epsilon^2}{2\pi \br^2} \int_0^T dt  \int_0^1 |\psi(\zeta) |^2
  \frac{|\delta(\frac{\zeta}{\sqrt{\epsilon}}+\frac{1}{\epsilon})|^2}{(\zeta\sqrt{\epsilon}+1)^2} \ d\zeta, \\
  &\leq  C \epsilon^2.
\end{align*}
%
Similarly we can show that 
$$\|\phi^\epsilon\|^2_{L^2(0,T;L^2(\mathbb{R}))} \leq C \epsilon^{4} .$$ 
\end{proof}

Let us give some more properties of $\sigma^\epsilon, v^\epsilon,\phi^\epsilon$ which we will need later on to prove the main results.
\begin{lem} \label{lem:improved-regularity}
 Let $(\sigma_T^\epsilon, v_T^\epsilon, \phi_T^\epsilon)$  and 
 $(\sigma^\epsilon(x,t), v^\epsilon(x,t),\phi^\epsilon(x,t))$ be
 as in Theorem \ref{thm4}. Then $\sigma_T^\epsilon, v_T^\epsilon$ and $ \phi_T^\epsilon$
 lie in $H^1(\mathbb{R}).$ Moreover $(\sigma^\epsilon, v^\epsilon,\phi^\epsilon) \in (H^1(0,T;H^1(\mathbb{R})))^3.$
\end{lem}

\begin{proof}
 Let us first show that $\sigma_T^\epsilon \in H^1(\mathbb{R}).$
 \begin{align*}
  \int_{\mathbb{R}} (1 +|\xi|^2)\hat\sigma_T^\epsilon(\xi) &= \int_{\mathbb{R}} (1 +|\xi|^2) \epsilon^{1/2} 
 \left|\psi\left(\sqrt{\epsilon}\left(\xi - \frac{1}{\epsilon}\right)\right)\right|^2 d\xi \\
 &\leq C\left(1 + \frac{1}{\epsilon^2}\right) \int_{\mathbb{R}}|\psi(\zeta)|^2 \  d\zeta \leq C\left(1 + \frac{1}{\epsilon^2}\right).
 \end{align*}
Thus $\sigma_T^\epsilon \in H^1(\mathbb{R}).$ In a similar way we can show that $v_T^\epsilon$ and $\phi_T^\epsilon$ 
belong to $ H^1(\mathbb{R}).$ Next we will show $\sigma^\epsilon  \in  H^1(0,T;H^1(\mathbb{R})).$  First we have
\begin{align*}
 \int_0^T\int_{\mathbb{R}} (1+\xi^2) |\hat\sigma^\epsilon(\xi,t)|^2 d\xi dt =  
 &\int_0^T\int_{\mathbb{R}} (1+\xi^2) |\sh_T^\epsilon(\xi)|^2 e^{(-2\delta(\xi)(T-t))} d\xi \ dt \\
&\leq C\left(1 + \frac{1}{\epsilon^2}\right) \int_0^T dt \int_{\mathbb{R}} |\psi(\zeta)|^2 \leq  CT\left(1 + \frac{1}{\epsilon^2}\right).
\end{align*}
Next we have 
\begin{align*}
 \int_0^T\int_{\mathbb{R}} (1+\xi^2) |\hat\sigma_t^\epsilon(\xi,t)|^2 d\xi dt =  
 &\int_0^T\int_{\mathbb{R}} (1+\xi^2) |\delta(\xi)|^2|\sh_T^\epsilon(\xi)|^2 e^{(-2\delta(\xi)(T-t))} d\xi \ dt \\
&\leq C\left(1 + \frac{1}{\epsilon^2}\right) \int_0^T dt \int_{\mathbb{R}} |\psi(\zeta)|^2 \leq  CT\left(1 + \frac{1}{\epsilon^2}\right).
\end{align*}
Thus we have shown $\sigma^\epsilon \in H^1(0,T;H^1(\mathbb{R})).$ Similarly we can show that $v^\epsilon$ and $\phi^\epsilon$ also belong to 
$H^1(0,T;H^1(\mathbb{R})).$
\end{proof}

\begin{lem}
Let $\sigma^\epsilon$ be as in Theorem \ref{thm4}. Then for any $\eta > 0,$ we have the following estimate  
 \begin{equation} \label{eq:sigmaestimate}
  \|\sigma^\epsilon(\cdot,0)\|^2_{L^2(|x - x_0|\leq \eta)}  \geq \frac{1}{4\pi} e^{-2a_1T}
 \end{equation}
\end{lem}
 \begin{proof}
 First by lemma \ref{lem:improved-regularity}, we conclude that $\sigma^\epsilon(\cdot,0) \in H^1(\mathbb{R}).$
 Proceeding in a similar way as in Theorem \ref{thm4}, we first obtain   
 \begin{equation}
  \|\sigma^\epsilon(\cdot,0)\|^2_{L^2(|x - x_0|\geq \eta)} \leq C \sqrt{\epsilon}
 \end{equation}
 for some positive  constant $C,$ independent of $\epsilon.$ Therefore
 \begin{equation} 
  \|\sigma^\epsilon(\cdot,0)\|^2_{L^2(|x - x_0|\leq \eta)} \geq \frac{1}{2\pi} e^{-2a_1T} -  C \sqrt{\epsilon} \geq \frac{1}{4\pi} e^{-2a_1T}
 \end{equation}
 for $\epsilon$ small.
\end{proof}

\begin{lem} \label{rem:trace}
Let $v^\epsilon$ and $\phi^\epsilon$ be as in Theorem \ref{thm4}. Then 
\begin{equation}
 \|v^\epsilon(L,\cdot)\|_{H^1(0,T)} \leq C \epsilon^{3/4} , \mbox{ and } \|\phi^\epsilon(L,\cdot)\|_{H^1(0,T)} \leq C \epsilon^{7/4},
\end{equation}
for any $L\in \mathbb{R}.$
\end{lem}
\begin{proof}
 By Lemma \eqref{lem:improved-regularity}, we have $v^\epsilon(L,\cdot)$ and $\phi^\epsilon(L,\cdot)$ lie in $H^1(0,T)$ for any 
 $L\in \mathbb{R}.$ We have 
 \begin{align*}
  |v^\epsilon(L,t)| \leq C \int_{\mathbb{R}}|\hat\sigma_T(\xi)|\frac{|\delta(\xi)|}{|\xi|}  d\xi 
  \leq C\epsilon^{3/4} \int_{\mathbb{R}} |\psi(\zeta)| d\zeta \leq C\epsilon^{3/4},
 \end{align*}
 and 
 \begin{align*}
  |v^\epsilon_t(L,t)| \leq C \int_{\mathbb{R}}|\hat\sigma_T(\xi)|\frac{|\delta(\xi)|^2}{|\xi|}  d\xi 
  \leq C\epsilon^{3/4} \int_{\mathbb{R}} |\psi(\zeta)| d\zeta \leq C\epsilon^{3/4}.
 \end{align*}
Thus $$ \|v^\epsilon(L,\cdot)\|_{H^1(0,T)}\leq C\epsilon^{3/4}.$$
 Similarly, we can show  
\begin{align*}
 \|\phi^\epsilon(L,\cdot)\|_{H^1(0,T)} \leq C \epsilon^{7/4}.
\end{align*}
\end{proof}

\subsection{Proof Of Main Theorems.}
Now we will use the above construction to prove Theorem \ref{thm1}. First we prove the following theorem. Theorem \ref{thm1} will be a 
direct consequence of this theorem.


\begin{thm} \label{thm5}
Let $\mathcal{O}_1$ be a proper subset of $(0,L)$ and $\mathcal{O}_2 \subseteq (0,L), \mathcal{O}_3 \subseteq (0,L).$ 
Then there exists a terminal condition $(\sigma_T, v_T, \phi_T) \in Z,$ such that the solution of \eqref{eq:adjoint1}, corresponding to this terminal condition, satisfy the following estimates
\begin{itemize}
 \item [(i)] 
 \begin{equation*}
  \|\sigma(\cdot,0)\|_{L^2(0,L)} \geq \frac{1}{8\pi} e^{-2a_1T},
 \end{equation*}
 and 
\item[(ii)]
\begin{equation}
\|\sigma\|_{L^2(0,T;L^2(\mathcal{O}_1))}^2 + \| v \|_{L^2(0,T;L^2(\mathcal{O}_2))}^2  + \| \phi\|_{L^2(0,T;L^2(\mathcal{O}_3))}^2 
\leq C \sqrt{\epsilon}.
\end{equation}
\end{itemize}
\end{thm}

\begin{proof}
As $\mathcal{O}_1$ is a proper subset of $(0,L),$
we choose
$x_0 $ and $\eta>0  $ such that 
$$\{x : |x - x_0| \leq \eta \} \subset \ip \mbox{ and does not intersect  } \mathcal{O}_1. $$ 
Let us choose $\hat\sigma_T^\epsilon(\xi)$ is as in Theorem \ref{thm4} with the above choice $x_0.$
Let $(\sigma_T^\epsilon, v^\epsilon_T, \phi_T^\epsilon)$ and $(\sigma^\epsilon, v^\epsilon, \phi^\epsilon)$ are as in \eqref{eq:adjoint_condn}
and \eqref{eq:representation} respectively. 

Let us  set
\begin{align}
 &q^\epsilon_0(t) = v^\epsilon(0,t) , \qquad q^\epsilon_L(t) = v^\epsilon(L,t), \notag \\
 &r^\epsilon_0(t) = \phi^\epsilon(0,t) , \qquad r^\epsilon_L(t) = \phi^\epsilon(L,t). 
\end{align}
Let $(\breve \sigma^\epsilon(x,t),\breve v^\epsilon (x,t) ,\breve \phi^\epsilon (x,t))$ be the restriction of 
$(\sigma^\epsilon(x,t), v^\epsilon(x,t),\phi^\epsilon(x,t))$ to $(0,L) \times (0,T)$ and 
$(\breve\sigma^\epsilon_T,\breve v^\epsilon_T, \breve\phi^\epsilon_T)$ be the restriction of 
$(\sigma^\epsilon_T, v^\epsilon_T, \phi^\epsilon_T)$ to $(0,L).$

Then  $\breve\sigma^\epsilon(x,t)$, $\breve v^\epsilon(x,t)$ and
$\breve\phi^\epsilon(x,t)$ satisfy the following system, 
\begin{equation}  
 \begin{array}{lll}
-  \breve \sigma_t^\epsilon (x,t) - \br \;\breve v^\epsilon_x \;\; = 0, \mbox{ in }  \ip \times (0,T),\\ [2.mm]
\displaystyle
 -  \breve v_t^\epsilon - \nu_0 \breve v^\epsilon_{xx}  - \frac{R \bt}{\br}\;\breve \sigma^\epsilon_x  - 
 R \breve \phi^\epsilon_x  \; = 0,  \mbox{ in }  \ip \times (0,T),\\[3.mm]
\displaystyle
 - \breve \phi_t^\epsilon - k_0 \breve \phi^\epsilon_{xx} - \frac{R \bt}{c_v} \breve v^\epsilon_x \; = 0,  \mbox{ in }  \ip \times (0,T),
\\ [2.mm] \displaystyle
  \breve \sigma^\epsilon(T) = \breve\sigma^\epsilon_T,\quad \breve v^\epsilon(T) = \breve v^\epsilon_T ,\quad 
  \breve \phi^\epsilon(T) = \breve \phi^\epsilon_T,
  \\ [2.mm]
  \breve v^\epsilon(0,t) = q^\epsilon_0(t), \quad  \breve v^\epsilon(L,t) = q^\epsilon_L(t) \quad \forall \ \ t > 0,
 \\ [2.mm]
 \breve\phi^\epsilon(0,t) = r^\epsilon_0(t), \quad  \breve\phi^\epsilon(L,t) = r^\epsilon_L(t) \quad \forall \ \ t > 0.
\end{array}
\end{equation}
Note that by Lemma \ref{lem:improved-regularity} and Lemma \ref{rem:trace}, we have 
$(\breve\sigma^\epsilon_T,\breve v^\epsilon_T, \breve\phi^\epsilon_T) \in (H^1(0,L))^3$ and 
$(q_0^\epsilon,q_L^\epsilon,r_0^\epsilon,r_L^\epsilon) \in (H^1(0,T))^4.$ The above system has a unique solution
$(\breve\sigma^\epsilon, \breve v^\epsilon, \breve\phi^\epsilon)$ with 
$(\breve\sigma^\epsilon, \breve v^\epsilon, \breve\phi^\epsilon) \in C([0,T];Z).$ They satisfy the following estimates:
\begin{itemize}
 \item[(i)] From \eqref{eq:sigmaestimate}, we have 
 \begin{equation}
  \|\breve\sigma^\epsilon(\cdot,0)\|_{L^2(0,L)} \geq \|\breve\sigma^\epsilon(\cdot,0)\|_{L^2(|x - x_0|\leq \eta )}
  \geq \frac{1}{4\pi}e^{-2 a_1 T}.
 \end{equation}
\item[(ii)] From Theorem \ref{thm4}, 
\begin{align}
 &\|\breve \sigma^\epsilon\|^2_{L^2(0,T;L^2(\mathcal{O}_1))} \leq \|\breve\sigma^\epsilon\|^2_{L^2(0,T;L^2(|x-x_0|\geq \eta)}
 \leq C\sqrt{\epsilon}, \notag \\
  &\|\breve v^\epsilon\|^2_{L^2(0,T;L^2(\mathcal{O}_2))} \leq \|v^\epsilon\|^2_{L^2(0,T;L^2(\mathbb{R}))} \leq C 
  \epsilon^{2}, \notag \\
&\|\breve \phi^\epsilon\|^2_{L^2(0,T;L^2(\mathcal{O}_3))} \leq C \epsilon^{4}.
 \end{align}
\item[(iii)] From Lemma \ref{rem:trace},
\begin{equation} \label{eq:traceestimate}
\|q_0^\epsilon\|_{H^1(0,T)} \leq C \epsilon^{3/4}, \ \|q_1^\epsilon\|_{H^1(0,T)} \leq C \epsilon^{3/4},  
\|r_0^\epsilon\|_{H^1(0,T)} \leq C \epsilon^{7/4}, \ \|r_1^\epsilon\|_{H^1(0,T)} \leq C \epsilon^{7/4}.
\end{equation}
\end{itemize}

Let $\widetilde \sigma^\epsilon (x,t)$, $\widetilde v^\epsilon(x,t)$ and $\widetilde \phi^\epsilon(x,t)$ satisfy the following
system 
\begin{equation}  
 \begin{array}{lll}
-  \widetilde\sigma_t^\epsilon - \br \;\widetilde v^\epsilon_x \;\; = 0,  \mbox{ in }  \ip \times (0,T),\\ [2.mm]
\displaystyle
 - \widetilde v_t^\epsilon (x,t)- \nu_0 \widetilde v^\epsilon_{xx}  - \frac{R \bt}{\br}\;\widetilde\sigma^\epsilon_x  - 
 R \widetilde\phi^\epsilon_x  \; = 0,  \mbox{ in }  \ip \times (0,T), \\[3.mm]
\displaystyle
 -  \widetilde\phi_t^\epsilon  - k_0 \widetilde\phi^\epsilon_{xx} - \frac{R \bt}{c_v} \widetilde v^\epsilon_x \; = 0, 
 \mbox{ in }  \ip \times (0,T),
\\ [2.mm] \displaystyle
  \widetilde \sigma^\epsilon(T) = 0,\quad \widetilde v^\epsilon(T) = 0 ,\quad 
  \widetilde \phi^\epsilon(T) = 0,  \mbox{ in }  \ip,
  \\ [2.mm]
  \widetilde v^\epsilon(0,t) = - q^\epsilon_0(t), \quad  \widetilde v^\epsilon(L,t) = - q^\epsilon_L(t) \quad \forall \ \ t > 0,
 \\ [2.mm]
 \widetilde \phi^\epsilon(0,t) =  - r^\epsilon_0(t), \quad  \widetilde \phi^\epsilon(L,t) = - r^\epsilon_L(t) \quad \forall \ \ t > 0.
\end{array}
\end{equation}
As $(q_0^\epsilon,q_L^\epsilon,r_0^\epsilon,r_L^\epsilon) \in (H^1(0,T))^4,$ the above system has a unique solution 
$(\widetilde \sigma^\epsilon, \widetilde v^\epsilon, \widetilde \phi^\epsilon)^T
\in C([0,T];Z).$ Using \eqref{eq:traceestimate}, we  obtain 
\begin{align}
 \|(\widetilde \sigma^\epsilon, \widetilde v^\epsilon, \widetilde \phi^\epsilon)^T\|^2_{L^2(0,T;Z)} 
 &\leq C
 (\|q^\epsilon_0\|_{L^2(0,T)} + \|q^\epsilon_L\|_{L^2(0,T)} + \|r^\epsilon_0\|_{L^2(0,T)} + \|r^\epsilon_L\|_{L^2(0,T)}) \notag \\
& \leq C \epsilon ^{3/4}. 
 \end{align}
 and 
\begin{align}
 \|\widetilde \sigma^\epsilon(\cdot,0)\|_{L^2(0,L)} \leq C \|\widetilde \sigma^\epsilon \|_{C([0,T)];L^2(0,L))} \leq C \epsilon^{3/4}.
\end{align}

Let us set
\begin{align}
  \sigma(x,t) = \breve\sigma^\epsilon(x,t) +  \widetilde \sigma^\epsilon(x,t), \notag \\
  v(x,t) = \breve v^\epsilon(x,t) +  \widetilde v^\epsilon(x,t), \notag \\
 \phi(x,t) = \breve \phi^\epsilon(x,t) +  \widetilde \phi^\epsilon(x,t). 
\end{align}
Then, $(\sigma, v,\phi)$ satisfy the system \eqref{eq:adjoint1} and the estimate holds.  
\end{proof}

\noindent Using the above theorem now we can prove Theorem \ref{thm1}.

{\bf Proof of Theorem \ref{thm1}.}
\begin{proof}
Theorem \ref{thm5} shows, in the observability inequality 
\eqref{eq:observability_ineq_1}, L.H.S $> \displaystyle \frac{1}{8\pi}e^{-2a_1T}$ and R.H.S $\leq C\sqrt{\epsilon}.$ Hence \eqref{eq:observability_ineq_1}
cannot hold. Hence system 
 \eqref{eq:linearized} - \eqref{eq:ini+bdy}
 is not null controllable. 
\end{proof}

\begin{remark} \label{rem:lessRegControl}
 The above result is established when $f\in L^2(0,T;L^2(\mathcal{O}_1))$, $g\in L^2(0,T;L^2(\mathcal{O}_2))$ and 
 $h\in L^2(0,T;L^2(\mathcal{O}_3))$. We can extend this negative result to less regular control. More precisely we can take 
 $g \in L^2(0,T;H^{-1}(\mathcal{O}_2))$ and $h \in L^2(0,T;H^{-1}(\mathcal{O}_3)).$ In order to give a 
 sense to the R.H.S. of \eqref{eq:identity_0}, we need to replace 
 $\chi_{\mathcal{O}_2}$ and $\chi_{\mathcal{O}_3}$ in \eqref{eq:linearized} by $\Psi_2 \in C^\infty_c(\mathcal{O}_2)$ and 
 $\Psi_3 \in C^\infty_c(\mathcal{O}_3)$
 respectively. The observability inequality becomes
 \begin{align} \label{eq:observability_ineq_4}
   &\|\sigma(\cdot,0)\|^2_{L^2(0,L)} + \|v(\cdot,0)\|^2_{L^2(0,L)}  + \|\phi(\cdot,0)\|^2_{L^2(0,L)}  
  \notag \\
  &\leq C \left ( \int_0^T \int_{{\mathcal O}_1} \sigma^2 dx dt + \int_0^T \int_{{\mathcal O}_2} (\Psi_2v)_x^2 dx dt
  + \int_0^T \int_{{\mathcal O}_3} (\Psi_3\phi)_x^2 dx dt \right )
\end{align}
where $(\sigma,v,\phi)$ is the solution of the adjoint system \eqref{eq:adjoint1}.
 Using the same construction as above one can prove that this  observability inequality does not hold and hence the system is not
 null controllable. 
 \end{remark}

\begin{remark} \label{rem:bdyControl}
One can use the above Gaussian Beam construction to rule out null controllability using a boundary control. Let us consider the system
\eqref{eq:linearized} - \eqref{eq:ini+bdy} with $f = g =h = 0$ and $v(L,t) = q(t) \in L^2(0,T).$ In this case null controllability 
is equivalent to the following observability inequality
\begin{align} \label{eq:observability_ineq_bdy}
  \|\sigma(\cdot,0)\|_{L^2(0,L)}^2 + \|v(\cdot,0)\|_{L^2(0,L)}^2 + \|\phi(\cdot,0)\|_{L^2(0,L)}^2 \leq C 
  \int_0^T |R \br\bt \sigma(L,t) + \br^2\nu_0 v_x(L,t)|^2 \ dt.
 \end{align}
where $(\sigma,v,\phi)$ is the solution of the adjoint system \eqref{eq:adjoint1}. We can use the above construction to show that 
the above observability 
inequality does not hold and hence the system is not null controllable by a boundary control in any time $T > 0$.
 \end{remark}

 We proved that the system \eqref{eq:linearized} - \eqref{eq:ini+bdy} is not null controllable when initial condition
lies in $Z$. So the natural question is if the system is null controllable or not when the initial conditions are regular.
In case of barotropic fluid, if initial density lies in $H^1_m(0,L)$ then the system is null controllable
(See \cite{CRR} and \cite{CMRR2}). In this section we choose $(\rho_0,u_0,\theta_0) \in H^1_m(0,L) \times L^2(0,L) \times L^2(0,L)$. We will
first show that \eqref{eq:linearized} - \eqref{eq:ini+bdy} is null controllable by velocity and temperature control only 
(i.e. when $f \equiv 0$) acting everywhere 
in the domain with this initial regular condition. Then we will show that we cannot achieve null 
controllability by localizing velocity and temperature control.

We consider the following interior control system 
\begin{equation} \label{eq:interior_control_1}
 \begin{array}{lll}
 \rho_t  + \br \;u_x \;\; = 0, \mbox{ in } (0,L) \times (0,T), \\ [2.mm]
\displaystyle
 u_t - \nu_0 u_{xx}  + \frac{R \bt}{\br}\;\rho_x  + R \theta_x  \; = g \chi_{\mathcal{O}_2}, \mbox{ in } (0,L) \times (0,T), \\[3.mm]
\displaystyle
 \theta_t  - k_0 \theta_{xx} + \frac{R \bt}{c_v} u_x \; = h \chi_{\mathcal{O}_3},\ \mbox{ in } (0,L) \times (0,T),
\\ [2.mm] \displaystyle
  \rho(0) = \rho_0,\quad u(0) = u_0 ,\quad \theta(0) = \theta_0, \mbox{ in } (0,L) 
  \\ [2.mm]
  u(0,t) = 0 = u(L,t) \quad \forall \ \ t > 0, \quad \theta(0,t) = 0 = \theta(L,t) \quad \forall \ \ t > 0.
\end{array}
 \end{equation}
Here $g$ and $h$ are velocity and temperature control respectively. Let us first explain why we need average zero condition for initial density.
Integrating the density equation of \eqref{eq:interior_control_1} in $(0,L)$ and using the boundary conditions we deduce that
\begin{equation*}
\frac{d}{dt} \left(\int _0 ^L \rho(x,t) dx  \right) = 0.
\end{equation*}
Therefore
\begin{equation}
 \int_0^L \rho(x,T) dx =   \int_0^L \rho_0(x) dx.
\end{equation}
Thus if the system \eqref{eq:interior_control_1} is null controllable in time $T > 0$
  then necessarily
  \begin{equation}
  \int_0^L \rho_0(x) dx = 0.
\end{equation}

Let us define $$ V = H^1_m(0,L) \times L^2(0,L) \times L^2(0,L) .$$
We have the following lemma about existence and uniqueness of solution to the system \eqref{eq:interior_control_1} follows easily from semigroup
theory. 
\begin{lem}
 Given $(\rho_0,u_0,\theta_0) \in V,$ $g \in L^2(0,T;L^2(0,L))$  and  $h \in L^2(0,T;L^2(0,L)),$ the 
 system \eqref{eq:interior_control_1} has 
 a unique solution  $(\rho,u,\theta) $ with $\rho \in L^2(0,T;H^1_m(0,L))$ and $(u,\theta) \in (L^2(0,T;H^1_0(0,L)))^2.$  Moreover
 $(\rho,u,\theta)$ belongs to $C([0,T];V).$
\end{lem}

%

{\bf Proof of Theorem \ref{thm2}.}
\begin{proof}
 Let us first take the following system with interior control $\tilde g$
 \begin{equation} \label{eq:thm7_1}
 \begin{array}{lll}
 \rho_t  + \br \;u_x \;\; = 0, \mbox{ in } (0,L) \times (0,T),\\ [2.mm]
\displaystyle
 u_t  - \nu_0 u_{xx}  + \frac{R \bt}{\br}\;\rho_x   \; = \tilde g, \mbox{ in } (0,L) \times (0,T), \\[3.mm]
 \displaystyle
  \rho(0) = \rho_0,\quad u(0) = u_0 , \mbox{ in } (0,L),
  \\ [2.mm]
  u(0,t) = 0 = u(L,t) \quad \forall \ \ t > 0.
\end{array}
 \end{equation}
 By Theorem 5.1 of \cite{CRR}, we know that for every $(\rho_0,u_0) \in H^1_m(0,L) \times L^2(0,L),$ there exists a control
 $\tilde g \in L^2(0,T;L^2(0,L)),$
 such that the solution of \eqref{eq:thm7_1} satisfies
 \begin{equation}
  (\rho,u)(x,T) = 0 , \mbox{ for all } x \in \ip. 
 \end{equation}
Now we consider the following heat equation
\begin{equation} \label{eq:thm7_2}
 \begin{array}{lll}
\displaystyle
 \theta_t  - k_0 \theta_{xx}  = \tilde h(x,t) , \mbox{ in } (0,L) \times (0,T),
\\ [2.mm] \displaystyle
 \theta(x,0) = \theta_0(x) , x \in \ip, \ \ 
 \theta(0,t) = 0 = \theta(L,t) \quad \forall \ \ t > 0.
\end{array}
 \end{equation}
 By Theorem 2.66 of \cite{CORON},  for every $\theta_0 \in L^2(0,L)$ 
 there exists a control $\tilde h \in L^2(0,T;L^2(0,L))$ such that the solution of \eqref{eq:thm7_2} satisfies 
 \[
  \theta(x,T) = 0 \mbox{ for all } x\in \ip.
 \]
Now define 
\[
\displaystyle
 g(x,t) = \tilde g(x,t) + R \theta_x(x,t) , \quad h(x,t) = \tilde h(x,t) + \frac{R\bt}{c_v} u_x(x,t),
\]
where $u$, $\theta$ are the solutions of \eqref{eq:thm7_1} and \eqref{eq:thm7_2} respectively. Thus $(\rho,u, \theta)$ is the solution 
of the system  \eqref{eq:interior_control_1} with $g$ and $h$ defined as above and it satisfies
$$ (\rho,u,\theta)(x,T) = 0. $$ Therefore the Theorem follows.
\end{proof}

Now we want to show that even if $(\rho_0, u_0, \theta_0) \in V,$ null controllability cannot be achieved by localized 
velocity and temperature controls, i.e., we want to prove Theorem \ref{thm3}.
First we will derive an observability inequality. In order to do this we introduce, 
\begin{equation}
 \alpha (x,t) = \rho_x(x,t).
\end{equation}
Differentiating the first equation of $\eqref{eq:interior_control_1},$ we obtain the following system 
\begin{equation} \label{eq:interior_control_2}
 \begin{array}{lll}
 \alpha_t  + \br \;u_{xx} \;\; = 0, \mbox{ in } (0,L) \times (0,T), \\ [2.mm]
\displaystyle
u_t - \nu_0 u_{xx}  + \frac{R \bt}{\br}\;\alpha  + R \theta_x  \; = g \chi_{\mathcal{O}_2}, \mbox{ in } (0,L) \times (0,T), \\[3.mm]
\displaystyle
\theta_t (x,t) - k_0 \theta_{xx} + \frac{R \bt}{c_v} u_x \; = h \chi_{\mathcal{O}_3}, \mbox{ in } (0,L) \times (0,T),
\\ [2.mm] \displaystyle
  \alpha(0) = \alpha_0 := (\rho_0)_x,\quad u(0) = u_0 ,\quad \theta(0) = \theta_0, \mbox{ in } (0,L),
  \\ [2.mm]
  u(0,t) = 0 = u(L,t) \quad \forall \ \ t > 0, \quad 
 \theta(0,t) = 0 = \theta(L,t) \quad \forall \ \ t > 0.
\end{array}
 \end{equation}
Here $(\alpha_0,u_0,\theta_0) \in Z$. The system \eqref{eq:interior_control_2} is well posed in $Z$. Note that, to prove Theorem \ref{thm3}, it
is enough to show system \eqref{eq:interior_control_1} is not null controllable in $Z$ by localized interior controls $g$ and $h$. 
As before we have the following proposition,

\begin{prop}
 For every $(\alpha_0, u_0, \theta_0) \in Z,$ the system \eqref{eq:interior_control_2} is null controllable in time $T$ by localized interior
 controls $g \in L^2(0,T;L^2(\mathcal{O}_2))$ and $h \in L^2(0,T;L^2(\mathcal{O}_3))$ if and only if for every 
 $(\sigma_T, v_T, \phi_T) \in Z,$ the solution of the following adjoint system 
 \begin{equation}  \label{eq:adjoint5}
 \begin{array}{lll}
- \sigma_t  + \br \;v \;\; = 0, \mbox{ in } (0,L) \times (0,T),\\ [2.mm]
\displaystyle
-  v_t - \nu_0 v_{xx}  + \frac{R \bt}{\br}\;\sigma_{xx}  - R \phi_x  \; = 0, \mbox{ in } (0,L) \times (0,T), \\[3.mm]
\displaystyle
-  \phi_t  - k_0 \phi_{xx} - \frac{R \bt}{c_v} \phi_x \; = 0, \mbox{ in } (0,L) \times (0,T),
\\ [2.mm] \displaystyle
  \sigma(T) = \sigma_T ,\quad v(T) = v_T ,\quad \phi(T) = \phi_T, \mbox{ in } (0,L),
  \\ [2.mm]
  \sigma(0,t) = 0 = \sigma(L,t) \quad \forall \ \ t > 0, \quad 
  v(0,t) = 0 = v(L,t) \quad \forall \ \ t > 0,
 \\ [2.mm]
 \phi(0,t) = 0 = \phi(L,t) \quad \forall \ \ t > 0.
\end{array}
\end{equation}
satisfies
\begin{align}  \label{eq:observability_ineq_3}
   &\|\sigma(\cdot,0)\|^2_{L^2(0,L)} + \|v(\cdot,0)\|^2_{L^2(0,L)}  + \|\phi(\cdot,0)\|^2_{L^2(0,L)}  
  \notag \\
  &\leq C \left (  \int_0^T \int_{{\mathcal O}_2} v^2 dx dt
  + \int_0^T \int_{{\mathcal O}_3} \phi^2 dx dt \right ).
\end{align}
\end{prop}
The system \eqref{eq:adjoint5} is well posed in $Z$. 
Let us consider the adjoint problem in $\mathbb{R} \times (0,T)$
\begin{equation}  \label{eq:adjoint6}
 \begin{array}{lll}
- \sigma_t + \br \;v \;\; = 0,  \mbox{ in } \mathbb{R} \times (0,T),\\ [2.mm]
\displaystyle
-  v_t - \nu_0 v_{xx}  + \frac{R \bt}{\br}\;\sigma_{xx} - R \phi_x  \; = 0, \mbox{ in } \mathbb{R} \times (0,T), \\[3.mm]
\displaystyle
- \phi_t - k_0 \phi_{xx} - \frac{R \bt}{c_v} \phi_x \; = 0, \mbox{ in } \mathbb{R} \times (0,T),
\\ [2.mm] \displaystyle
  \sigma(T) = \sigma_T ,\quad v(T) = v_T ,\quad \phi(T) = \phi_T, \mbox{ in } \mathbb{R}.
\end{array}
\end{equation}
Applying the Fourier transform in \eqref{eq:adjoint6}, we obtain the following system of ODE
\begin{align}
&- \begin{pmatrix} \sh \\ \vh \\ \ph \end{pmatrix}_t = \bar A(\xi) \begin{pmatrix} \sh \\ \vh \\ \ph \end{pmatrix}, \notag \\
&(\sh,\vh,\ph)(\xi,T) = (\sh_T, \vh_T, \ph_T), 
\end{align}
where 
\begin{equation}
\displaystyle
 \bar A(\xi) = \begin{pmatrix} 
                        \displaystyle  0                       & - \br                & 0 \\
                        \displaystyle  \frac{R\bt}{\br}\xi^2  & - \nu_0 \xi^2                 & Ri\xi \\
                        \displaystyle   0                      &\displaystyle \frac{R \bt}{c_v} i \xi & -k_0\xi^2
                        \end{pmatrix}
\end{equation}
Eigenvalues of $\bar A(\xi)$ are the same as eigenvalues of $A(\xi)$. Let $-\delta(\xi)$ be the eigenvalue satisfying 
$\lim_{|\xi| \rightarrow \infty} \delta(\xi) = \omega_0.$
The eigenfunction of $\bar A(\xi)$ corresponding to $-\delta(\xi)$
is $\displaystyle\left(1, \frac{\delta(\xi)}{\br}, -i \xi d_\delta(\xi) \right),$ where $d_\delta(\xi)$ is as in Lemma \ref{lem:eigenfunction}. 
Thus if we choose 
$$(\sh_T, \vh_T, \ph_T)(\xi) = \sh_T(\xi) \displaystyle\left(1, \frac{\delta(\xi)}{\br}, -i \xi d_\delta(\xi) \right), $$
the solution of \eqref{eq:adjoint6} can be written in the following way 
\begin{align} \label{eq:representation1}
 & \sigma(x,t) = \frac{1}{2 \pi} \int_{\mathbb{R}} \sh_T(\xi) e^{i x \xi} e^{-\delta(\xi) (T-t)} \ d\xi, \notag \\
 & v(x,t) = \frac{1}{2\pi} \int_{\mathbb{R}} \sh_T(\xi) e^{i x \xi} \frac{ \delta(\xi)}{\br } e^{-\delta(\xi) (T-t)} \ d\xi, \notag \\
& \phi(x,t) = \frac{1}{2\pi} \int_{\mathbb{R}} \sh_T(\xi) e^{i x \xi} (-i \xi d_{\delta}(\xi)) e^{-\delta(\xi) (T-t)} \ d\xi. 
 \end{align}
We have the following Theorem.
\begin{thm} \label{thm8}
 Let $\sh^\epsilon_T$ be  as in Theorem \ref{thm4} and 
 $(\sigma^\epsilon(x,t), v^\epsilon(x,t),\phi^\epsilon(x,t))$  be as in  \eqref{eq:representation1}. They satisfy the following
\begin{itemize}
 \item[(i)] $\frac{1}{2\pi} e^{-2 a_1 T} \leq \|\sigma^\epsilon(\cdot,0)\|^2_{L^2(\mathbb{R})} 
 \leq \frac{1}{2\pi} $
 \item[(ii)] For any $\eta > 0$ ,  there exists a constant $C$ independent of $\epsilon$ such that 
 \begin{equation}
  \|\sigma^\epsilon\|^2_{L^2(0,T;L^2(|x - x_0|\geq \eta))} \leq C \sqrt{\epsilon}, \quad
  \|v^\epsilon\|^2_{L^2(0,T;L^2(|x - x_0|\geq \eta))} \leq C \sqrt \epsilon
 \end{equation}
\item[(iii)] 
\begin{equation} 
 \|\phi^\epsilon\|^2_{L^2(0,T;L^2(\mathbb{R}))} \leq C \epsilon^{2} 
\end{equation}
for some positive constant $C$ independent of $\epsilon$. 
\end{itemize}
\end{thm}

\begin{proof}
 The proof is similar to Theorem \ref{thm4}.
\end{proof}

Now we can proceed in a similar way as in Theorem \ref{thm5} to prove Theorem \ref{thm3}.
%

\section{Null Controllability of  Compressible Non-Barotropic Navier Stokes System in One Dimension Linearized about 
\texorpdfstring{$(\bar \rho, \bar v, \bt).$}{Lg}}

In this section we will discuss null controllability of system \eqref{eq:blinearized1}. We want to prove Theorem \ref{thm1.4}. 
First we have the following Proposition about existence and uniqueness of the system \eqref{eq:blinearized1}.
\begin{prop}
 Let $(\rho_0,u_0,\theta_0) \in Z$. Let us assume that $f \in L^2(0,T;L^2(\mathcal{O}_1))$, \\ $g \in L^2(0,T;L^{2}(\mathcal{O}_2))$
 and $h \in L^2(0,T;L^{2}(\mathcal{O}_3)).$ Then \eqref{eq:blinearized1} has a unique solution $(\rho,u,\theta)$
 with $\rho \in L^2(0,T;L^2(0,L))$, $u \in L^2(0,T;H^1_0(0,L))$ and $\theta \in L^2(0,T;H^1_0(0,L)).$ Moreover $(\rho,u,\theta)$ belongs to
 $C([0,T];Z).$
\end{prop}
Proceeding as before we have
the following proposition about the equivalence of null controllability and observability inequality. 
\begin{prop}
 For every initial state $(\rho_0,u_0,\theta_0) \in (L^2(0,L))^3,$ the system \eqref{eq:blinearized1} is null controllable 
 in any time $T$
 by localized interior 
 controls $f \in L^2(0,T;L^2(\mathcal{O}_1)),$  $g \in L^2(0,T;L^2(\mathcal{O}_2))$ and $h \in L^2(0,T;L^2(\mathcal{O}_3))$
 if and only if for every 
 $(\sigma_T, v_T, \phi_T) \in (L^2(0,L))^3,$ $(\sigma,v,\phi),$ the solution of the following adjoint system 
 \begin{equation} \label{eq:badjoint}
\begin{array}{l}
\displaystyle
- \sigma_t - \bar v \sigma_x - \br \;v_x \;\; = 0, \mbox{ in } \ip \times (0,T),\\ [2.mm]
\displaystyle
-  v_t (x,t)- \nu_0 v_{xx}  - \frac{R \bt}{\br}\;\sigma_x  - \bar v v_x - 
R \phi_x  \; = 0, \mbox{ in } \ip \times (0,T), \\[3.mm]
\displaystyle
- \phi_t -\frac{k}{\br c_v} \phi_{xx} - \frac{R \bt}{c_v} v_x - \bar v \phi_x(x,t) \; = 0, \mbox{ in } \ip \times (0,T),
\\[3.mm] \displaystyle
\sigma(T) = \sigma_T,\quad \quad   v(T)= v_T,  \mbox{ and } \phi(T) = \phi_T, \qquad \mbox{ in } \ip \times (0,T),,
\\ [2.mm] \displaystyle
\sigma(L,t) = 0, \;\;\; v(0,t) = 0 = v(L,t), \;\;  t \in (0,T),
\\ [2.mm] \displaystyle
\phi(0,t) = 0 = \phi(L,t), \;\;  t \in (0,T).
\end{array}
\end{equation}
satisfies the following observability inequality
\begin{align} \label{eq:bobservability}
&\|\sigma(\cdot,0)\|^2_{L^2(0,L)} + \|v(\cdot,0)\|^2_{L^2(0,L)}  + \|\phi(\cdot,0)\|^2_{L^2(0,L)}  
  \notag \\
  &\leq C \left ( \int_0^T \int_{{\mathcal O}_1} \sigma^2 dx dt + \int_0^T \int_{{\mathcal O}_2} v^2 dx dt
  + \int_0^T \int_{{\mathcal O}_3} \phi^2 dx dt \right ).
\end{align}
\end{prop}

The above adjoint system is well posed in $(L^2(0,L))^3$. We want to show the the observability inequality  
\eqref{eq:bobservability} does not hold for small time $T$. Let us first 
consider the adjoint problem \eqref{eq:badjoint}, in $\mathbb{R}\times(0,T)$ as a terminal value problem only 
\begin{equation} \label{eq:badjointR}
 \begin{array}{l}
\displaystyle
- \sigma_t  - \bar v \sigma_x  - \br \;v_x \;\; = 0,\mbox{ in } \mathbb{R} \times (0,T), \\ [2.mm]
\displaystyle
- v_t - \nu_0 v_{xx}  - \frac{R \bt}{\br}\;\sigma_x  - \bar v v_x - 
R \phi_x  \; = 0, \mbox{ in } \mathbb{R} \times (0,T), \\[3.mm]
\displaystyle
-  \phi_t  -\frac{k}{\br c_v} \phi_{xx} - \frac{R \bt}{c_v} v_x - \bar v \phi_x \; = 0, \mbox{ in } \mathbb{R} \times (0,T),
\\[3.mm] \displaystyle
\sigma(T) = \sigma_T,\quad \quad   v(T)= v_T,  \mbox{ and } \phi(T) = \phi_T, \qquad \mbox{ in } \mathbb{R}.
\end{array}
\end{equation}

We have the following theorem.
\begin{thm}
Let $(\sigma_T^\epsilon, v_T^\epsilon, \phi_T^\epsilon)$ be as in Theorem \ref{thm4}.
 Let $(\sigma^\epsilon, v^\epsilon, \phi^\epsilon)$ be the solution of \eqref{eq:badjointR} with this terminal condition.
 Then they satisfy the following estimates
 \begin{itemize}
  \item [(i)] $\displaystyle \frac{1}{2\pi} e^{-2a_1T} \leq \|\sigma^\epsilon(\cdot,0)\|_{L^2(\mathbb{R})} \leq \frac{1}{2\pi} .$
  \item [(ii)] For any $\eta > 0$, there exists a constant $C$ independent of $\epsilon$ such that 
  \begin{equation}
   \|\sigma^\epsilon\|^2_{L^2(0,T;L^2(|x-x(t)|\geq \eta))} \leq C \sqrt{\epsilon},
  \end{equation}
where $x(t) = x_0 - \bar v(T -t).$
\item[(iii)] 
There exists a constant $C,$ independent of $\epsilon,$ such that 
\begin{equation}
 \|v^\epsilon\|^2_{L^2(0,T;L^2(\mathbb{R}))} \leq C \epsilon^2, \qquad \|\phi^\epsilon\|^2_{L^2(0,T;L^2(\mathbb{R}))} \leq C \epsilon^4.
\end{equation}
\end{itemize}
\end{thm}

\begin{proof}
Let $(\sigma,v,\phi)$ be a solution of \eqref{eq:badjointR}.
Let us define the transformed functions for $(x,t) \in \mathbb{R} \times (0,T)$
$$\tilde\sigma(x,t) = \sigma(x-\bar v (T-t),t),\  \tilde v(x,t) = v(x-\bar v (T-t),t),\  \tilde\phi(x,t) = \phi(x-\bar v (T-t),t).$$
Then $(\tilde\sigma,\tilde v,\tilde \phi)$ is a solution of  the system \eqref{eq:adjoint2}. 
Thus from Theorem \ref{thm4}, we easily prove the theorem.  
\end{proof}

\noindent  We  now give the proof of Theorem \ref{thm1.4}. 

{\bf Proof of Theorem \ref{thm1.4}.}
 
 \begin{proof}
 We will proceed in a similar way as in Theorem \ref{thm5}. In order to obtain a cotradiction to the observability inequality 
 \eqref{eq:bobservability}, we need $(x(t) -\eta, x(t)+\eta) \subset (0,L)$ and does not intersect the set 
 $\mathcal{O}_1 = (l_1,l_2)$ for all $t.$ In particular
   $ (x_0 -\eta - \bar v T, x_0+\eta - \bar v T)$ must lie inside $(0,L)$ and does not intersect $(l_1,l_2).$ Thus  we have to choose 
   $x_0$ and $\eta$ properly so that 
  $ (x_0 -\eta - \bar v T, x_0+\eta - \bar v T)$ is either a proper subset of $(0,l_1)$ or $(l_2,L)$. 

   Let us first choose $x_0 \in (0,l_1)$ and $\eta >0$ such that $(x_0 -\eta, x_0+\eta) \subset (0, l_1)$.  
   Thus if we want $(x_0 -\eta - \bar v T, x_0+\eta - \bar v T)$ to be a proper subset of $(0,l_1)$, 
   we need $x_0 - \eta - \bar v T > 0.$ This implies $\displaystyle T < \frac{l_1}{\bar v}.$ Similarly if we choose $x_0 \in (l_2,L)$
   then we need $\displaystyle T < \frac{L-l_2}{\bar v}.$
  
  Thus when $\displaystyle T < \max\left\{ \frac{l_1}{\bar v}, \frac{L-l_2}{\bar v }\right\},$ the observability inequality 
  \eqref{eq:bobservability} does not hold and hence the system is not null controllable.
  \end{proof}



\section{Null Controllability of  Compressible Barotropic Navier Stokes System in Two Dimension Linearized about 
\texorpdfstring{$(\bar \rho, 0,0).$}{Lg}}
 In this section, we will prove Theorem \ref{thm1.5}. We will construct highly localized solutions, as we did in one dimension. Let us introduce 
 the following constants
 $$\displaystyle \mu_0 := \frac{\mu}{\bar \rho}, \quad \gamma_0 := \frac{\lambda + \mu}{\bar \rho}, 
 \quad b_1 := a \gamma \bar \rho^{\gamma -2} . $$
 The system
 \eqref{eq:linearized2D} is well posed in $(L^2(\Omega))^3.$ As before we have the equivalence between null controllability and observability 
 inequality

\begin{prop}
 For every initial state $(\rho_0,{\bf u}_0) \in (L^2(\Omega))^3,$ the system \eqref{eq:linearized2D} is null controllable by localized interior 
 controls $f \in L^2(0,T;L^2(\mathcal{O}_1))$ and ${\bf g} \in L^2(0,T;(L^2(\mathcal{O}_2))^2)$ if and only if for every 
 $(\sigma_T, {\bf v}_T) \in (L^2(\Omega))^3$ the solution of the following adjoint system 
 \begin{equation} \label{eq:2Dadjoint}
 \begin{array}{l}
\displaystyle
- \sigma_t (x,t) -  \bar \rho \ \mathrm{div}\;{\bf v} \;\;= 0, \mbox{ in } \Omega \times (0,T), 
\\[2.mm] \displaystyle
-  {\bf v}_t - \mu_0  \Delta {\bf v} - \gamma_0 \nabla \mathrm{div} \ {\bf v} 
 - b_1 \nabla \sigma \;=\;0 ,  \mbox{ in } \Omega \times (0,T),
\\[2.mm] \displaystyle
\sigma(T) = \sigma_T\quad \mbox{and}\quad   {\bf v}(T)= {\bf v}_T, \quad  \mbox{ in } \Omega,
\\ [2.mm] \displaystyle
 {\bf v} = {\bf 0}  \ \ \mbox{ on } \partial \Omega \times (0,T).
\end{array}
\end{equation}
satisfies the observability inequality
\begin{equation} \label{eq:2Dobservability}
 \|\sigma(\cdot,0)\|^2_{L^2(\Omega)} + \|{\bf v}(\cdot,0)\|^2_{(L^2(\Omega))^2} \leq C \left(\int_0^T \int_{\mathcal{O}_1} \sigma^2 dx dt 
  + \int_0^T \int_{\mathcal{O}_2} {\bf v}^2 dx dt \right ).
\end{equation}
\end{prop}
Here ${\bf v}(x,t) = (v_1,v_2)(x,t).$
The above adjoint system is well posed in $(L^2(\Omega))^3.$ We will now use ``Gaussian Beam '' construction  to show that the observability
inequality \eqref{eq:2Dobservability} does not hold. Let us first consider the adjoint problem \eqref{eq:2Dadjoint} in 
$\mathbb{R}^2 \times (0,T)$ as a terminal value problem only. Applying the Fourier transformation, we obtain the following system of ODE

\begin{align}
&- \begin{pmatrix} \sh \\ \vh_1 \\ \vh_2  \end{pmatrix}_t = \tilde A(\xi) \begin{pmatrix} \sh \\ \vh_1 \\ \vh_2  \end{pmatrix} ,
\;\;(\sh,\vh_1,\vh_2)(\xi,T) = (\sh_T, \vh_{1,T},\vh_{2,T}), \ \; (\xi, t) \in \mathbb{R}^2 \times (0,T) 
\end{align}
where 
\begin{equation}
\displaystyle
 \tilde A(\xi) = \begin{pmatrix}
                  0 & \br i\xi_1 & \br i\xi_2 \\ b_1 i \xi_1 & -\mu_0 |\xi|^2 - \gamma_0 \xi^2_1 & -\gamma_0 \xi_1 \xi_2 
                  \\ b_1 i \xi_2 & -\gamma_0 \xi_1 - \xi_2 & -\mu_0 |\xi|^2 - \gamma_0 \xi^2_2
                 \end{pmatrix}.
\end{equation}
We have the following lemma.
\begin{lem}
\begin{itemize}
 \item The eigenvalues of $\tilde A(\xi)$ are  
 $$\tilde \lambda(\xi) = -\mu_0 |\xi|^2,\;\;
 \tilde \mu(\xi) = -\frac{(\gamma_0+\mu_0)|\xi|^2}{2}\left(1 + \sqrt{1- \frac{4b_1\br}{(\gamma_0+\mu_0)|\xi|^2}}\right) $$ 
 $$\tilde \delta(\xi) = -\frac{(\gamma_0+\mu_0)|\xi|^2}{2}\left(1 -  \sqrt{1- \frac{4b_1\br}{(\gamma_0+\mu_0)|\xi|^2}}\right) $$
 \item For all $\xi \in \mathbb{R}^2,$ there exists a constant $ a_2> 0 $, such that 
 $$0 \leq - \mathrm{Re} \ \tilde \delta(\xi) < a_2. $$
 \item There exists $\xi_0 > 0,$ such that for all $|\xi| \geq \xi_0$ , $\tilde \lambda(\xi),\tilde \mu(\xi)$ and $\tilde \delta(\xi)$  
 are all real and distinct. The eigenvalues satisfy
 $$\displaystyle \lim_{|\xi|\rightarrow \infty} \frac{-\tilde \mu(\xi)}{|\xi|^2} = \mu_0 +\gamma_0,\qquad 
 \lim_{|\xi|\rightarrow \infty} {-\tilde \delta(\xi)} = \frac{b_1\bar \rho}{\mu_0 +\gamma_0}.\qquad  $$
 \item
  For $|\xi|\geq \xi_0,$ the eigenvalues are differentiable and  we have 
 $$\displaystyle \left| \bigtriangleup_{\xi} \tilde \delta(\xi) \right| \leq \frac{C}{|\xi|^2},$$
 for some positive constant $C$. 
 \item The eigenfunction of $\tilde A(\xi)$ corresponding to $\tilde \delta(\xi)$ is
 $\displaystyle\left(1, \frac {\tilde \delta(\xi) \xi_1}{\br i |\xi|^2}, \frac {\tilde \delta(\xi) \xi_2}{\br i |\xi|^2}\right).$
\end{itemize}
\end{lem}


Now if we choose $$\displaystyle (\sh_T,{\bf \vh}_T)(\xi) = \sh_T(\xi) 
\left(1, \frac {\tilde\delta(\xi) \xi_1}{\br i |\xi|^2}, \frac {\tilde \delta(\xi) \xi_2}{\br i |\xi|^2}\right),$$ the solution
of \eqref{eq:2Dadjoint} in $\mathbb{R}^2 \times (0,T),$ can be written in the following way 
\begin{align} \label{eq:2Drepresentation}
 & \sigma(x,t) = \frac{1}{(2\pi)^2} \int_{\mathbb{R}^2} \sh_T(\xi) e^{i x \xi}  e^{\tilde\delta(\xi)(T-t)} d\xi, \notag \\
 & v_1(x,t) = \frac{1}{(2\pi)^2} \int_{\mathbb{R}^2} \sh_T(\xi) e^{i x \xi} \ \frac {\tilde \delta(\xi) \xi_1}{\br i |\xi|^2} \
 e^{\tilde\delta(\xi)(T-t)} d\xi, \notag \\
 & v_2(x,t) = \frac{1}{(2\pi)^2} \int_{\mathbb{R}^2} \sh_T(\xi) e^{i x \xi} \ \frac {\tilde \delta(\xi) \xi_2}{\br i |\xi|^2} \
 e^{\tilde\delta(\xi)(T-t)} d\xi.
 \end{align}
 
 \begin{thm}
  Let $\bar \xi \in \mathbb{R}^2$ with $|\bar \xi | = 1$ and $x_0 \in \mathbb{R}^2.$ Let $\tilde\psi$ be a smooth function compactly supported
  in the unit ball and of unit $L^2({\mathbb{R}^2})$ norm. For any $\epsilon > 0$, define 
  $$ \sh_T^\epsilon(\xi) = \epsilon^{\frac{1}{4}}\tilde\psi\left(\sqrt{\epsilon}\left(\xi - \frac{\bar \xi}{\epsilon}\right) \right) e ^{-ix_0 \xi}$$
 and $(\sigma^\epsilon, v_1^\epsilon, v_2^\epsilon )$ as in \eqref{eq:2Drepresentation}. Then they satisfy the following estimates
\begin{itemize}
 \item [(i)] $\displaystyle \frac{1}{(2\pi)^2} e^{-2a_2T} \leq \|\sigma^\epsilon(\cdot,0)\|_{L^2(\mathbb{R}^2)} \leq \frac{1}{(2\pi)^2} .$
  \item [(ii)] For any $\eta > 0$, there exists a constant $C$ independent of $\epsilon$ such that 
  \begin{equation}
   \|\sigma^\epsilon\|^2_{L^2(0,T;L^2(|x-x_0|\geq \eta))} \leq C \sqrt{\epsilon},
  \end{equation}
\item[(iii)] There exists a positive constant $C$ independent of $\epsilon$ such that 
\begin{equation}
 \|{\bf v}^\epsilon\|_{L^2(0,T;(L^2(\mathbb{R}))^2)} \leq C \epsilon^2.
\end{equation}
\end{itemize}
\end{thm}

\begin{proof}
 We will only give proof of estimate (4.6). Other estimates can be proved in a similar manner as before. From \eqref{eq:2Drepresentation} we have
 \begin{align*}
 \sigma^\epsilon(x,t) &= \frac{1}{(2 \pi)^2}\int_{\mathbb{R}^2} \epsilon^{\frac{1}{4}} 
 \tilde\psi\left(\sqrt{\epsilon}\left(\xi - \frac{\bar \xi}{\epsilon}\right)\right)
 e^{i (x-x_0) \xi} e^{\tilde\delta(\xi)(T-t)} \ d\xi \\ 
 &= \frac{\epsilon^{-1/4}}{(2\pi)^2} \int_{|\zeta|\leq 1} \tilde\psi(\zeta) e^{i(x-x_0)(\frac{\zeta}{\sqrt{\epsilon}}+ 
\frac{\bar \xi}{\epsilon})}
 e^{\tilde\delta(\frac{\zeta}{\sqrt{\epsilon}}+ \frac{\bar \xi}{\epsilon})(T-t)} \ d\zeta.
\end{align*}
Note that 
\[
\displaystyle \bigtriangleup_{\zeta}  e^{i(x-x_0)(\frac{\zeta}{\sqrt{\epsilon}}+ \frac{\bar \xi}{\epsilon})}  = 
- \frac{|x-x_0|^2}{\sqrt{\epsilon}} e^{i(x-x_0)(\frac{\zeta}{\sqrt{\epsilon}}+ \frac{\bar \xi}{\epsilon})}.
\]
Thus for $|x-x_0| \geq \eta > 0,$ we have 
\begin{align}
 \sigma^\epsilon(x,t) &=  - \frac{\epsilon^{1/4}}{4\pi^2  |x-x_0|^2} \int_{|\zeta|\leq 1} 
 \bigtriangleup_{\zeta} \left( e^{i(x-x_0)(\frac{\zeta}{\sqrt{\epsilon}}+ \frac{\bar \xi}{\epsilon})} \right) \tilde\psi(\zeta)
  e^{\tilde\delta(\frac{\zeta}{\sqrt{\epsilon}}+ \frac{\bar \xi}{\epsilon})(T-t)} \ d\zeta \notag \\
  &= -\frac{\epsilon^{1/4}}{4\pi^2  |x-x_0|^2} \int_{|\zeta|\leq 1} 
   e^{i(x-x_0)(\frac{\zeta}{\sqrt{\epsilon}}+ \frac{\bar \xi}{\epsilon})}  
   \bigtriangleup_{\zeta} 
   \left( \tilde\psi(\zeta) e^{\tilde\delta(\frac{\zeta}{\sqrt{\epsilon}}+ \frac{\bar \xi}{\epsilon})(T-t)} \right) \ d\zeta.
\end{align}
 
Thus for $\epsilon$ small, we have 
\begin{align*}
 \left|\bigtriangleup_{\zeta} 
   \left( \tilde\psi(\zeta) e^{\tilde\delta(\frac{\zeta}{\sqrt{\epsilon}}+ \frac{1}{\epsilon})(T-t)} \right) \right| 
    \leq C \sqrt{\epsilon}.
\end{align*}
Notice that the constant $C$ is independent of $\epsilon.$ Therefore for $\epsilon$ small enough and $|x - x_0|\geq \eta >0$ we have 
\begin{align}
 |\sigma^\epsilon(x,t)| \leq C \frac{\epsilon^{1/4}}{4\pi^2|x- x_0|^2} \sqrt{\epsilon}.
\end{align}
Thus there exists a positive constant $C(T,\eta)$, such that 
\begin{align}
 \|\sigma^\epsilon\|^2_{L^2(0,T;L^2(|x-x_0|\geq \eta))} \leq C \sqrt{\epsilon}.
\end{align}
This proves (ii).
\end{proof}

\begin{remark}
 Now we can proceed as before to prove Theorem \ref{thm1.5}.
\end{remark}

\end{document}